\documentclass[11pt,a4paper]{amsart}
\usepackage{amsfonts}
\usepackage{}
\usepackage{amsfonts}
\usepackage{amsmath,xypic}
\usepackage{mathrsfs}
\usepackage{amssymb}

\usepackage{CJK,CJKnumb}
\usepackage[CJKbookmarks,colorlinks,
            linkcolor=black,
            anchorcolor=black,
            citecolor=blue]{hyperref}
\usepackage{color}              
\usepackage{indentfirst}        
\usepackage{latexsym,bm}        
\usepackage{amsmath,amssymb}    
\usepackage{pstricks}
\usepackage{pst-node}
\usepackage{pst-tree}
\usepackage{pst-plot}
\usepackage{pst-text}
\usepackage{graphicx}
\usepackage{cases}
\usepackage{pifont}
\usepackage{txfonts}
\usepackage[all,knot,poly]{xy}

\allowdisplaybreaks[4]  

\CJKtilde   

\newcommand{\R}{\mathscr{R}}

\newtheorem{theorem}{Theorem}[section]
\newtheorem{lemma}{Lemma}[section]
\newtheorem{remark}{Remark}[section]
\newtheorem{proposition}{Proposition}[section]
\newtheorem{definition}{Definition}[section]
\definecolor{black}{rgb}{0.00,0.00,0.45}

\allowdisplaybreaks[4]
                               { \end{spacing}
                               } 
\begin{document}
\title[Double-bosonization and Majid's Conjecture (IV)]{Double-bosonization and Majid's Conjecture, (IV): Type-Crossings from $A$ to $BCD$}
\author[H. Hu]{Hongmei Hu}
\address{Department of Mathematics,  Shanghai Key Laboratory of Pure Mathematics and Mathematical Practice,
East China Normal University,
Minhang Campus,
Dong Chuan Road 500,
Shanghai 200241,
PR China}
\email{hmhu0124@126.com}

\author[N. Hu]{Naihong Hu$^\ast$}
\address{Department of Mathematics,  Shanghai Key Laboratory of Pure Mathematics and Mathematical Practice,
East China Normal University,
Minhang Campus,
Dong Chuan
Road 500,
Shanghai 200241,
PR China}
\email{nhhu@math.ecnu.edu.cn}

\date{April 14, 2015}



\thanks{N.~H., the corresponding author, supported by the NNSFC (Grant No.
 11271131).}

\date{}
\maketitle

\newcommand*{\abstractb}[3]{ %
                             \begingroup%
                             \leftskip=8mm \rightskip=8mm
                             \fontsize{11pt}{\baselineskip}\noindent{\textbf{Abstract} ~}{#1}\\ 
                             {\textbf{Keywords} ~}{#2}\\ 
                             {\textbf{MR(2010) Subject Classification} ~}{#3}\\ 
                                                          \endgroup
                           }
\begin{abstract}
Both in Majid's double-bosonization theory and in Rosso's quantum shuffle theory, the rank-recursive and type-crossing construction for $U_q(\mathfrak g)$'s is
still a remaining open question.
Working in Majid's framework in this paper, based on the generalized double-bosonization Theorem we proved before,
we further describe explicitly the type-crossing construction of $U_q(\mathfrak g)$'s for
$(BCD)_n$ series directly from type $A_{n-1}$ via adding a pair of dual braided groups determined by a pair of $(R, R')$-matrices of type $A$
derived from the respective suitably chosen representations.
Combining with our results of the first three papers of this series,
this solves Majid's conjecture, that is, any quantum group $U_q(\mathfrak g)$
associated to a simple Lie algebra $\mathfrak g$ can be grown out of $U_q({\mathfrak {sl}}_2)$
recursively by  a series of suitably chosen double-bosonization procedures.
\end{abstract}

\bigskip
\bigskip

\section{Introduction and our results}
There are several well-known ways to understand the structure of
quantum groups, for example, Drinfeld's quantum double of any
finite-dimensional Hopf algebra introduced in \cite{dri1}, the
FRT-construction in \cite{FRT1} based on $R$-matrices limited to
the classical types. Afterwards, Majid rediscovered the Radford
biproduct, nowadays named ``Radford-Majid bosonization", in the
framework of a braided monoidal category of left modules over a Hopf algebra whose Drinfeld center is just the category of
Yetter-Drinfeld modules over it (see
\cite{radford}, \cite{rt}, \cite{majid8}, \cite{majid9}, and for the first proof of the latter see that of Example 1.3 in \cite{majid10}),
Majid then developed the double-bosonization theory, which
improved the FRT-construction and extended Drinfeld's quantum
double to generalized quantum double associated to a weakly
quasitriangular dual pair (not necessarily nondegenerate) of braided
groups coming from the (co)module category of a (co)quasitriangular
Hopf algebra. Another analogous construction in spirit was given by Sommerh\"auser in \cite{somm} in a larger category of
Yetter-Drinfeld modules, but the latter is seemly not suitable for the rank-recursive construction considered here.

Roughly speaking, the Majid's double-bosonization theory is given as
follows. Assigned to a pair of dual braided groups $B^{\star}, B$
covariant under a underlying quasitriangular Hopf algebra $H$, there
is a new quantum group on the tensor space $B^{\star}\otimes
H\otimes B$ by double-bosonization in \cite{majid1}, consisting of
$H$ extended by $B$ as additional `positive roots' and its dual
$B^{\star}$ as additional `negative roots'. On the other hand, Majid
claimed that the rank-recursive construction of $U_q(\mathfrak g)$'s
can be obtained in principle in his context, and expected his
double-bosonization allows to
generate a tree of quantum groups. 
But how to do it? Except two examples in low rank given by Majid
(see \cite{majid1,majid6,majid7}), it is still a remaining open
question for almost 20 years with a main challenge from representation theory to elaborate the full tree structure of
quantum groups generated by the double-bosonization procedures. The
technical difficulties mainly involve in the exact choices of
certain representations as well as treating with some nonstandard
$R$-matrices derived from these representations. In general, as we known,
it is difficult to capture some information encoded in a
(nonstandard) $R$-matrix associated to some representation of
$U_{q}(\mathfrak g)$, for example, the spectral decomposition of
$R$-matrix, to normalize $R$-matrices, etc. However, the authors have overcome
some technical difficulties to give concretely the rank-recursive
constructions of $U_q(\mathfrak g)$'s for the $ABCD$ series in
\cite{HH1}, the type-crossing constructions for types $F_4, G_2$ in
\cite{HH2} and the type-crossing rank-recursive constructions for
types $E_6, E_7, E_8$ in \cite{HH3}. Since Majid's framework on
braided groups in certain braided categories has been developed currently
into the framework of Nichols algebras in the
Yetter-Drinfeld categories for classifying finite-dimensional
pointed Hopf algebras due to Andruskiewitsch-Schneider and
Heckenberger et al (see \cite{AS1, AS2}, \cite{HS}, etc.), we believe in some
sense that the study on how the double-bosonization
procedure yields what kinds of new quantum groups is significant.
An interesting application of our constructions might be connected to a recent work of Cuntz and Lentner on Nichols algebras (see \cite{cl}).

Some examples in low ranks given in \cite{HH1} show that $B_{2}, C_{3}, D_{4}$ can be grown out of $U_{q}({\mathfrak {sl}}_n)$, for $n=2,\,3,\,4$, respectively.
Then, what about the general type-crossing construction of the $(BCD)_n$ series directly from type $A_{n-1}$? This is also a
remaining problem of Rosso in his quantum shuffle setting (see \cite{rosso}).
If this problem is solved,
we will confirm completely the Majid's conjecture on how step by step and directly from type $A_1$ to get all quantum groups $U_q(\mathfrak g)$'s of
the finite dimensional complex simple Lie algebras
$\mathfrak g$.
As a reminder, we mention that
Rosso's quantum shuffle theory in \cite{rosso} gives an intrinsic,
or functorial understanding of $U_q(\mathfrak g)$ from the braidings.
Rosso also asserted the rank-recursive or type-crossing constructions of all positive part $U_q(\mathfrak n)$'s existing in his quantum shuffle
setting and listed an interpreting recipe simply at the Lie theory level. However, for the completeness of both theories, it is desirable to know how to proceed them in the respective
contexts. This also stimulates us to attack this question first in the Majid's framework.

The paper is organized as follows.
In section 2,
we recall some basic facts about the FRT-construction,
the Majid's double-bosonization theorem and our generalized double-bosonization theorem (\cite{HH2}).
Section $3$ is devoted to exploring the general type-crossing constructions for the $(BCD)_n$ series directly from node $A_{n-1}$.
Firstly, based on the vector representation $T_{V}$ of $U_{q}(\mathfrak{sl}_{n})$ and the corresponding standard $R$-matrix,
through choosing a matched matrix $R'$ such
that $(R, R')$ derives a pair of dual braided groups $\tilde{V}^{\vee}(R),\tilde{V}(R)$ covariant under $U_{q}(\mathfrak{sl}_{n})$, which is different from those used in \cite{HH1},
we construct the expected quantum group $U_{q}({\mathfrak {so}}_{2n+1})$ of higher rank $1$. Note that our construction procedure here for type $B_n$ is a bit different from those of types $C_n$ and $D_n$.
Secondly, in order to construct the expected $U_{q}({\mathfrak {sp}}_{2n})$ and $U_{q}({\mathfrak {so}}_{2n})$ of higher rank $1$,
we consider the quantum symmetric square $sym^{2}V$ and the quantum second exterior power $\wedge^{2}V$ of the vector representation $T_{V}$ of $U_{q}(\mathfrak{sl}_{n})$.
We first analyse the $R$-matrices $R_{VV}$ corresponding to these representations,
and obtain the minimal polynomials of $PR_{VV}$ with our skillful techniques sufficiently exploiting the features of the representations we work with. We get
a pair of matrices $(R, R')$ satisfying $(PR{+}I)(PR^{\prime}{-}I)=0$,
and new mutually dual braided groups $V^{\vee}(R^{\prime},R_{21}^{-1}), V(R^{\prime},R)$ covariant under $U_{q}(\mathfrak{sl}_{n})$, and then to arrive at our required objects.
As a final consequence of our results including those in \cite{HH1, HH2, HH3}, we confirm that Majid's conjecture is true, that is, starting with $U_{q}(\mathfrak{sl}_{2})$,
we obtain step by step all $U_{q}(\mathfrak g)$'s associated with the finite-dimensional complex simple Lie algebras $\mathfrak g$
by a series of delicately selected double-bosonization procedures.
In section 4, we contrast Majid's double-bosonization with Rosso's quantum shuffle theory to deduce the same extended Dynkin diagrams of higher rank one,
and draw the tree structure of quantum groups $U_{q}(\mathfrak{g})$'s, for all the finite-dimensional complex simple Lie algebras $\mathfrak{g}$,
grown out of the source node $A_1$ inductively by a series of suitably chosen double-bosonization procedures we build.

\section{Preliminaries}
In this paper, let $k$ be the complex field, $\mathbb{R}$ the real
field, $E$ the Euclidean space $\mathbb{R}^{n}$ or a suitable
subspace. Denote by $\varepsilon_{i}$ the usual orthogonal unit
vectors in $\mathbb{R}^{n}$. Let $\mathfrak g$ be a
finite-dimensional complex semisimple Lie algebra with simple roots
$\alpha_{i}$, $\lambda_{i}$ the fundamental weight corresponding to
simple root $\alpha_{i}$. The Cartan matrix of $\mathfrak g$ is
$(a_{ij})$, where
$a_{ij}=\frac{2(\alpha_{i},\alpha_{j})}{(\alpha_{i},\alpha_{i})}$,
and $d_{i}=\frac{(\alpha_{i},\alpha_{i})}{2}$. Let $(H,\R)$ be a
quasitriangular Hopf algebra, where $\R=\R^{(1)}\otimes \R^{(2)}$ is
the universal $R$-matrix, $\R_{21}=\R^{(2)}\otimes \R^{(1)}$,
denote by $\Delta,\eta,\epsilon$ its coproduct, counit, unit, and by
$S$ its antipode. We shall use the Sweedler's notation: for $h\in
H$, $\Delta(h)=h_{1}\otimes h_{2}$. $H^{op} \ (H^{cop})$ denotes the
opposite (co)algebra structure of $H$, respectively.
$\mathfrak{M}_{H} \ ({}_{H}\mathfrak{M}$) denotes the braided
category consisting of right (left) $H$-modules. If there exists a
coquasitriangular Hopf algebra $A$ such that the dual pair $(H,A)$
is a weakly quasitriangular, then $\mathfrak{M}_{H} \
({}_{H}\mathfrak{M}$) is equivalent to the braided category
${}^{A}\mathfrak{M} \ (\mathfrak{M}^{A})$ consisting of left (right)
$A$-comodules. For the detailed description of these, we left the
readers to refer to Drinfeld's and Majid's papers \cite{dri2},
\cite{majid2}, \cite{majid3}, and so on. By a braided group we mean
a braided bialgebra or Hopf algebra in some braided category. In
order to distinguish from the ordinary Hopf algebras, denote by
$\underline{\Delta},\underline{S}$ its coproduct and antipode,
respectively.

\subsection{FRT-construction}
Let $R$ be an invertible matrix obeying the quantum Yang-Baxter equation.
There is a bialgebra $A(R)$ \cite{FRT1} corresponding to the $R$-matrix,
called the FRT-bialgebra.
\begin{definition}
$A(R)$ is generated by $1$ and $T=\{t^{i}_{j}\}$,
having the following defining relations:
$$RT_{1}T_{2}=T_{2}T_{1}R,\quad
~~\Delta(T)=T\otimes T,\quad
\epsilon(T)=I,\quad
\mbox{where~}
T_{1}=T\otimes I,\quad
T_{2}=I\otimes T.$$
$A(R)$ is a coquasitriangular bialgebra with
$\R:~A(R)\otimes A(R) \longrightarrow k$ such that $\R(t^{i}_{j}\otimes t^{k}_{l})=R^{ik}_{jl}$,
where $R^{ik}_{jl}$ denotes the entry at row $(ik)$ and column $(jl)$ in the matrix $R$.
\end{definition}

On the other hand, $\Delta$ in $A(R)$ induces the multiplication in $A(R)^{\ast}=\text{Hom}(A(R),k)$.
 In \cite{FRT1}, $U_{R}$ is defined to be the subalgebra of $A(R)^{\ast}$ generated by $L^{\pm}=(l^{\pm}_{ij})$,
with relations
$(PRP)L_{1}^{\pm}L_{2}^{\pm}=L_{2}^{\pm}L_{1}^{\pm}(PRP),
$
$
(PRP)L_{1}^{+}L_{2}^{-}=L_{2}^{-}L_{1}^{+}(PRP),
$
where $P$ is the permutation matrix $P: P(u\otimes v)=v\otimes u$, $l^{\pm}_{ij}$ is defined by
$
(l^{+}_{ij},t^{k}_{l})=R^{ki}_{lj},
$
$
(l^{-}_{ij},t^{k}_{l})=(R^{-1})^{ik}_{jl}.
$
Moreover,
the algebra $U_{R}$ is a bialgebra with coproduct
$\Delta(L^{\pm})=L^{\pm}\otimes L^{\pm}$, and counit $\varepsilon(l_{ij}^{\pm})=\delta_{ij}$ (also cf. \cite{klim}).

Specially, when $R$ is one of the classical $R$-matrices, bialgebra
$A(R)$ has a quotient coquasitriangular Hopf algebra, denoted by
$Fun(G_{q})$ or $\mathcal{O}_{q}(G)$, and then $U_{R}$ has a
corresponding quotient quasitriangular Hopf algebra, which is
isomorphic to the extended quantized enveloping algebra
$U_{q}^{\text{ext}}(\mathfrak{g})$ (with groups-like elements indexed in a refined weight lattice in comparison with $U_q(\mathfrak{g})$). Moreover, there exists a
(non-degenerate) dual pairing $\langle ,\, \rangle$ between
$\mathcal{O}_{q}(G)$ and $U_{q}^{\text{ext}}(\mathfrak{g})$. The way
of getting the resulting quasitriangular algebras
$U_{q}^{\text{ext}}(\mathfrak{g})$ is the so-called FRT-construction
of the quantized enveloping algebras (for the classical types).

\begin{remark}\label{note}
The $R$-matrices used in Majid's paper \cite{majid1} are a bit different from those as in \cite{FRT1}, which are the
conjugations $P\circ\cdot\circ P$ of the ordinary $R$-matrices.
It is easy to check that $A(P\circ R\circ P)=A(R)^{\text{op}}$, where $(P\circ R\circ P)^{ij}_{kl}=R^{ji}_{lk}$.
Note that we will use the Majid's notation for $R$-matrices as in \cite{majid1} in the remaining sections of this paper.
\end{remark}

\subsection{Majid's double-bosonization}
Majid \cite{majid1} proposed the concept of a weakly quasitriangular dual pair via his insight on more examples on
matched pairs of bialgebras or Hopf algebras in \cite{majid5}.
This allowed him to establish a theory of double-bosonization in a broad framework that generalized the FRT's construction
which was limited to the classical types.

\begin{definition}
Let $(H,A)$ be a pair of Hopf algebras equipped with a dual pairing $\langle ,\,\rangle$
and convolution-invertible algebra/anti-coalgebra maps $\R,\bar{\R}: A \rightarrow H$ obeying
$$\langle\bar{\R}(a),b\rangle=\langle\R^{-1}(b),a\rangle,\quad
\partial^{R}h=\R \ast(\partial^{L}h)\ast\R^{-1},\quad
\partial^{R}h=\bar{\R}\ast(\partial^{L}h)\ast\bar{\R}^{-1}
$$
for
$a, b\in A, \ h\in H$.
Here $\ast$ is the convolution product in $\text{hom}(A,H)$ and
$(\partial^{L}h)(a)=\langle h_{(1)},a\rangle h_{(2)},
$
$
(\partial^{R}h)(a)=h_{(1)}\langle h_{(2)},a\rangle$
are left, right ``differentiation operators" regarded as maps $A \rightarrow H$ for fixed $h$.
\end{definition}

For any quasitriangular Hopf algebra $(H, \R)$  (where $\R=\R^{(1)}\otimes \R^{(2)}$), Majid \cite{majid8} rediscovered a special form of the Radford biproduct in the braided category ${}_{H}\mathfrak{M}$ of left $H$-modules,
currently named the ``Radford-Majid bosonization". Let us recall the definition.
\begin{definition}
Let $B$ be a braided group in ${}_{H}\mathfrak{M}$,
then its bosonization is the Hopf algebra $B\rtimes H$ defined as $B\otimes H$ (as a vector space) with the product, coproduct and antipode
\begin{gather*}
(b\otimes h)(c\otimes g)=b(h_{(1)}\rhd c)\otimes h_{(2)}g,\\
\Delta(b\otimes h)=b_{\underline{(1)}}\otimes \R^{(2)}h_{(1)}\otimes \R^{(1)}\rhd b_{\underline{(2)}}\otimes h_{(2)},\\
S(b\otimes h)=(Sh_{(2)})u\R^{(1)}\rhd\underline{S}b\otimes S(\R^{(2)}h_{(1)});\quad u\equiv(S\R^{(2)})\R^{(1)}.
\end{gather*}
On the other hand,
the right-handed version for $B\in \mathfrak{M}_{H}$ is $H\ltimes B$ defined by
\begin{gather*}
(h\otimes b)(g\otimes c)=hg_{(1)}\otimes (b\lhd g_{(2)})c,\\
\Delta(h\otimes b)=h_{(1)}\otimes b_{\underline{(1)}}\lhd \R^{(1)}\otimes h_{(2)}\R^{(2)}\otimes b_{\underline{(2)}},\\
S(h\otimes b)=S(h_{(2)}\R^{(2)})\otimes \underline{S}b\lhd \R^{(1)}vSh_{(1)};\quad v\equiv\R^{(1)}(S\R^{(2)}).
\end{gather*}
The corresponding formulae for bosonizations in the comodule categories $\mathfrak{M}^{A},\,{}^{A}\mathfrak{M}$, respectively
are trivially obtained by the usual conversions of the module formulae.
\end{definition}
Let $C, B$ be a pair of braided groups in $\mathfrak{M}_{H}$, which are called
dually paired if there is an intertwiner
$ev: C\otimes B \longrightarrow k$ such that
$
\text{ev}(cd,b)=\text{ev}(d,b_{\underline{(1)}})ev(c,b_{\underline{(2)}}),$
$
\text{ev}(c,ab)=\text{ev}(c_{\underline{(2)}},a)ev(c_{\underline{(1)}},b),$
$
\forall a,b\in B,c,d\in C.
$
Then $C^{op/cop}$ (with opposite product and coproduct) is a Hopf algebra in $_{H}\mathfrak{M}$,
which is dual to $B$ in the sense of an ordinary dual pairing $\langle~,~\rangle$,
which is $H$-bicovariant: $\langle h\rhd c,b\rangle=\langle c, b\lhd h\rangle$ for all $h\in H$.
Let $\overline{C}=(C^{op/cop})^{\underline{cop}}$,
then $\overline{C}$ is a braided group in $_{\overline{H}}\mathfrak{M}$,
where $\overline{H}$ is $(H,\R_{21}^{-1})$.
So there are two Hopf algebras $H\ltimes B$ (bosonization) and $\bar{C}\rtimes H$ (bosonization) by the above Definition.
What relations do there exist between them?
Majid gave the following double-bosonization theorem.
\begin{theorem}\label{ml1} $(${\rm\bf Majid}$)$
On the tensor space $\bar{C}\otimes H \otimes B$,
there is a unique Hopf algebra structure $U=U(\bar{C},H,B)$ such that
$H\ltimes B$ (bosonization) and $\bar{C}\rtimes H$ (bosonization) are sub-Hopf algebras by the canonical inclusions and
$$
\left.
\begin{array}{rl}
bc=&(\R_{1}^{(2)}\rhd c_{\overline{(2)}})
\R_{2}^{(2)}\R_{1}^{-(1)}(b_{\underline{(2)}}\lhd\R_{2}^{-(1)})\,
\langle\R_{1}^{(1)}\rhd c_{\overline{(1)}}, b_{\underline{(1)}}\lhd
\R_{2}^{(1)}\rangle\,\cdot\\
&\qquad\qquad\qquad\quad\cdot\,\langle\R_{1}^{-(2)}\rhd \overline{S}c_{\overline{(3)}}, b_{\underline{(3)}}\lhd \R_{2}^{-(2)}\rangle,
\end{array}
\right.
$$
for all $ b\in B, \,c\in\overline{C}$ viewed in $U$.
Here $\R_{1},\R_{2}$
are distinct copies of the quasitriangular structure $\R$ of $H$.
The product, coproduct of $U$ are given by
$$
\left.
\begin{array}{rl}
(c\otimes h\otimes b)\cdot(d\otimes g\otimes a)=&c(h_{(1)}\R_{1}^{(2)}\rhd d_{\overline{(2)}})
\otimes h_{(2)}\R_{2}^{(2)}\R_{1}^{-(1)}g_{(1)}\otimes (b_{\underline{(2)}}\lhd\R_{2}^{-(1)}g_{(2)})a\\
&\langle\R_{1}^{(1)}\rhd d_{\overline{(1)}},b_{\underline{(1)}}\lhd\R_{2}^{(1)}\rangle
\langle\R_{1}^{-(2)}\rhd\overline{S}d_{\overline{(3)}},b_{\underline{(3)}}\lhd\R_{2}^{-(2)}\rangle;
\end{array}
\right.
$$
$$
\Delta(c\otimes h\otimes b)=c_{\overline{(1)}}\otimes \R^{-(1)}h_{(1)}\otimes b_{\underline{(1)}}\lhd\R^{(1)}
\otimes \R^{-(2)}\rhd c_{\overline{(2)}}\otimes h_{(2)}\R^{(2)}\otimes b_{\underline{(2)}}.
$$
Moreover,
the antipodes of $H\ltimes B$ and $\bar{C}\rtimes H$ can be extended to an antipode $S: U \rightarrow U$ by
the two extensions
$S(chb)=(Sb)\cdot(S(ch))$
and
$S(chb)=(S(hb))\cdot(Sc)$.
\end{theorem}
\begin{remark}\label{Cross}
If there exists a coquasitriangular Hopf algebra $A$ such that $(H,A)$ is a weakly quasitriangular dual pair,
and
$b,
c$
are primitive elements,
then some relations simplify to
$$
[b,c]=\R(b^{\overline{(1)}})\langle c,b^{\overline{(2)}}\rangle-
\langle c^{\overline{(1)}},b\rangle\bar{\R}(c^{\overline{(2)}});\eqno{(C2)}
$$
$$
\Delta b=b^{\overline{(2)}}\otimes \R(b^{\overline{(1)}})+1\otimes b,\quad
\Delta c=c\otimes 1+\bar{\R}(c^{\overline{(2)}})
\otimes c^{\overline{(1)}}.\eqno{(C3)}
$$
\end{remark}
To each $R$, Majid associated two braided groups $V(R',R)$, $V^{\vee}(R',R_{21}^{-1})$
in the braided categories ${}^{A(R)}\mathfrak{M}, \ \mathfrak{M}^{A{(R)}}$,
called the braided vector algebra, the braided covector algebra, respectively in \cite{majid4} (see below).
\begin{proposition}\label{prop1}
Suppose that $R'$ is another matrix such that
$
(i) \ R_{12}R_{13}R'_{23}=R'_{23}R_{13}R_{12},
$
$
R_{23}R_{13}R'_{12}=R'_{12}R_{13}R_{23},
$
$(ii) \ (PR+1)(PR'-1)=0$,
$(iii) \ R_{21}R'_{12}=R'_{21}R_{12}$,
where $P$ is a permutation matrix with the entry $P^{ij}_{kl}=\delta_{il}\delta_{jk}$.
Then the braided-vector algebra $V(R',R)$ defined by generators $1$,
$\{e^{i}\mid i=1,\cdots,n\}$,
 and relations
$$e^{i}e^{j}=\sum\limits_{a,b} R'{}^{ji}_{ab}e^{a}e^{b}$$
forms a braided group with
$$\underline{\Delta}(e^{i})=e^{i}\otimes 1+1\otimes e^{i},\quad
\underline{\epsilon}(e^{i})=0,\quad
\underline{S}(e^{i})=-e^{i},\quad
\Psi(e^{i}\otimes e^{j})=\sum\limits_{a,b}R^{ji}_{ab}e^{a}\otimes e^{b}$$
in braided category ${}^{A(R)}\mathfrak{M}$.

Under the duality $\langle f_{j},e^{i}\rangle=\delta_{ij}$, the
braided-covector algebra $V^{\vee}(R',R_{21}^{-1})$ defined by $1$
and $\{f_{j}\mid j=1,\cdots,n\}$, and relations
$$f_{i}f_{j}=\sum\limits_{a,b}f_{b}f_{a}R'{}^{ab}_{ij}$$
forms another braided group with
$$\underline{\Delta}(f_{i})=f_{i}\otimes 1+1\otimes f_{i},\quad
\underline{\epsilon}(f_{i})=0,\quad
\underline{S}(f_{i})=-f_{i},\quad
\Psi(f_{i}\otimes f_{j})=\sum\limits_{a,b}f_{b}\otimes f_{a}R^{ab}_{ij}$$
in braided category $\mathfrak{M}^{A(R)}$.
\end{proposition}
\begin{remark}\label{rem1}
$PR$ obeys some minimal polynomial equation  $ \prod_{i}(PR-x_{i})=0. $ For
each nonzero eigenvalue $x_{i}$, we can normalize $R$ so that
$x_{i}=-1$. Then set $ R'=P+P\prod\limits_{j\neq i}(PR-x_{j}). $
This satisfies the conditions $(i)-(iii)$, which gives us at least
one mutually dual pair of braided (co)vector algebras for each
nonzero eigenvalue of $PR$.
\end{remark}

\subsection{Generalized double-bosonization Theorem}
Majid obtained another bialgebra $\widetilde{U(R)}$ by the double cross product bialgebra of $A(R)^{op}$ in \cite{majid5},
generated by $m^{\pm}$ with
$$
Rm^{\pm}_{1}m^{\pm}_{2}=m^{\pm}_{2}m^{\pm}_{1}R,
\quad
Rm_{1}^{+}m_{2}^{-}=m_{2}^{-}m_{1}^{+}R,
\quad
\Delta((m^{\pm})^{i}_{j})=(m^{\pm})_{j}^{a}\otimes (m^{\pm})_{a}^{i},
\quad \epsilon((m^{\pm})^{i}_{j})=\delta_{ij},$$
where
$(m^{\pm})^{i}_{j}$ denotes the entry at row $i$ and column $j$ in $m^{\pm}$.
Moreover,
$(\widetilde{U(R)}, A(R))$ is a weakly quasitriangular dual pair with
$$\langle(m^{+})^{i}_{j},t^{k}_{l}\rangle= R^{ik}_{jl},
\quad
\langle (m^{-})^{i}_{j},t^{k}_{l}\rangle =(R^{-1}){}^{ki}_{lj},
\quad\R(T)=m^{+},\quad\bar{\R}(T)=m^{-}.$$

When the $R$-matrix $R_{VV}$ we work with is nonstandard (By ``standard" we mean that the $R$-matrix is obtained by the vector representation of $U_q(\mathfrak g)$
when $\mathfrak g$ is limited to $ABCD$ series) (see section 3), we obtained the following 
\begin{theorem}
\cite{HH2} When the $R$-matrix $R_{VV}$ is nonstandard, the above weakly quasitriangular dual pair can be descended to a mutually dual pair of
Hopf algebras
$(U_{q}^{\text{ext}}(\mathfrak g), H_{\lambda R})$ with the correct modification:
$$
\langle (m^{+})^{i}_{j},t^{k}_{l}\rangle=R_{VV}{}^{ik}_{jl}=\lambda R^{ik}_{jl},\quad
\langle (m^{-})^{i}_{j},t^{k}_{l}\rangle=R_{VV}^{-1}{}^{ki}_{lj}=\lambda^{-1}(R^{-1}){}^{ki}_{lj}.
$$
$$
\langle (m^{+})^{i}_{j},\tilde{t}^{k}_{l}\rangle
=(R_{VV}^{t_{2}})^{-1}{}^{ik}_{jl}=((\lambda R)^{t_{2}})^{-1}{}^{ik}_{jl},\quad
\langle (m^{-})^{i}_{j},\tilde{t}^{k}_{l}\rangle
=[(R^{-1}_{VV})^{t_{1}}]^{-1}{}^{lj}_{ki}=[(\lambda^{-1}R^{-1})^{t_{1}}]^{-1}{}^{lj}_{ki}.
$$
Such $\lambda$ is called a {\it normalization constant of quantum groups}.
The convolution-invertible algebra\,/\,anti-coalgebra maps $\R,\bar{\R}$ in
$\textrm{hom}(H_{\lambda R},U_{q}^{\textrm{ext}}(\mathfrak{g}))$ are
$$
\R(t^{i}_{j})=(m^{+})^{i}_{j},\quad
\R(\tilde{t}^{i}_{j})=(m^{+})^{-1}{}^{j}_{i};\quad
\bar{\R}(t^{i}_{j})=(m^{-})^{i}_{j},\quad
\bar{\R}(\tilde{t}^{i}_{j})=(m^{-})^{-1}{}^{j}_{i},
$$
where $U_{q}^{\text{ext}}(\mathfrak g)$ is the FRT-form of $U_{q}(\mathfrak{g})$ or extended quantized enveloping algebra
adjoined by the elements $K_{i}^{\frac{1}{m}}$,
$H_{\lambda R}$ is the Hopf algebra associated with $A(R_{VV})$,
generated by $t^{i}_{j}, \tilde{t}^{i}_{j}$ and subject to
certain relations.
\end{theorem}
With the weakly quasitriangular dual pair of Hopf algebras in the above Theorem,
we obtain that the braided (co)vector algebras
$V(R^{\prime},R)\in{}^{H_{\lambda R}}\mathfrak{M}$
and
$V^{\vee}(R^{\prime},R_{21}^{-1})\in \mathfrak{M}^{H_{\lambda R}}$.
However, in order to yield the required quantum group of higher rank $1$, we need centrally extend the pair $(U_{q}^{\text{ext}}(\mathfrak g),H_{\lambda R})$
to the pair
$$
\Bigl(\,\widetilde{U_q^{\text{ext}}(\mathfrak
g)}=U_{q}^{\text{ext}}(\mathfrak g)\otimes k[c,c^{-1}],
\,\widetilde{H_{\lambda R}}=H_{\lambda R}\otimes k[g,g^{-1}]\,\Bigr)
$$
with the action $e^{i}\lhd c=\lambda e^{i}$,
$f_{i}\lhd c=\lambda f_{i}$ and the extended pair
$\langle c,g\rangle=\lambda$. In this way, we easily see that
the braided algebras $V(R^{\prime},R)\in{}^{\widetilde{H_{\lambda R}}}\mathfrak{M}$
and
$V^{\vee}(R^{\prime},R_{21}^{-1})\in \mathfrak{M}^{\widetilde{H_{\lambda R}}}$.

With these, we have the following {\it generalized double-bosonization Theorem} in \cite{HH2} by Theorem \ref{ml1}.
\begin{theorem}\label{cor1}
Let $R_{VV}$ be the $R$-matrix corresponding to the irreducible representation $T_{V}$ of $U_{q}(\mathfrak{g})$.
There exists a normalization constant $\lambda$ such that $\lambda R=R_{VV}$.
Then the new quantum group
$U=U(V^{\vee}(R^{\prime},R_{21}^{-1}),\widetilde{U_{q}^{ext}(\mathfrak g)},V(R^{\prime},R))$ has the following the cross relations:
\begin{gather*}
cf_{i}=\lambda f_{i}c,\quad
e^{i}c=\lambda ce^{i},\quad
[c,m^{\pm}]=0,\quad
[e^{i},f_{j}]=\delta_{ij}\frac{(m^{+})^{i}_{j}c^{-1}-c(m^{-})^{i}_{j}}{q_{\ast}-q_{\ast}^{-1}};\\
e^{i}(m^{+})^{j}_{k}=R_{VV}{}^{ji}_{ab}(m^{+})^{a}_{k}e^{b},\quad
(m^{-})^{i}_{j}e^{k}=R_{VV}{}^{ki}_{ab}e^{a}(m^{-})^{b}_{j},\\
(m^{+})^{i}_{j}f_{k}=f_{b}(m^{+})^{i}_{a}R_{VV}{}^{ab}_{jk},\quad
f_{i}(m^{-})^{j}_{k}=(m^{-})^{j}_{b}f_{a}R_{VV}{}^{ab}_{ik},
\end{gather*}
and the coproduct: $$\Delta c=c\otimes c, \quad \Delta
e^{i}=e^{a}\otimes (m^{+})^{i}_{a}c^{-1}+1\otimes e^{i}, \quad
\Delta f_{i}=f_{i}\otimes 1+c(m^{-})^{a}_{i}\otimes f_{a},$$ and the
counit $\epsilon e^{i}=\epsilon f_{i}=0$.
\end{theorem}
We can normalize $e^{i}$ such that the factor
$q_{\ast}-q_{\ast}^{-1}$ in Theorem \ref{cor1} meets the situation
we need. On the other hand, the FRT-generators $(m^{\pm})^{i}_{j}$
can be obtained by the following lemma, which was proved in
\cite{HH2}.
\begin{lemma}\label{lem1}
The elements $l_{ij}^{\pm}\in U_{q}(\mathfrak{g})$ are the $L$-functionals associated with some representation for $U_{q}(\mathfrak{g})$.
Then the FRT-generators $(m^{\pm})^{i}_{j}$ can be obtained by
$(m^{\pm})^{i}_{j}=S(l_{ij}^{\pm})$,
where $S$ is the antipode.
\end{lemma}

\section{Type-crossing construction of quantum groups for $BCD$ series}
Besides the results in \cite{HH1,HH2,HH3}, the remaining cases for the type-crossing
construction of $U_{q}(\mathfrak g)$'s needed to be considered are the $(BCD)_n$ series
directly starting from type $A_{n-1}$. The
authors will consider them for the first time in this section.
The above $R$-matrix datum
$R_{VV}$ is given by the universal $R$-matrix of
$U_h(\mathfrak g)$ and a certain representation $T_{V}$. Set $
\mathfrak{R}=\sum\limits_{r_{1},\cdots,r_{n}=0}^{\infty}\prod\limits_{j=1}^{n}
\frac{(1-q_{\beta_{j}}^{-2})^{r_{j}}}{[r_{j}]_{q_{\beta_{j}}}!}q_{\beta_{j}}^{\frac{r_{j}(r_{j}+1)}{2}}E_{\beta_{j}}^{r_{j}}\otimes
F_{\beta_{j}}^{r_{j}}$ (see \cite{klim}). Then $R_{VV}=B_{VV}\circ(T_{V}\otimes
T_{V})(\mathfrak{R})$, where $B_{VV}$ denotes the linear operator on
$V\otimes V$ given by $B_{VV}(v\otimes
w):=q^{(\mu,\mu^{\prime})}v\otimes w$, for $v\in V_{\mu},$ $w\in
V_{\mu^{\prime}}$.

In the sequel, we will take a basis of $V$ with weight-raising indices such that operators
$T_{V}(E_{i})$'s and $T_{V}(F_{i})$'s raise and lower the indices
of the basis respectively, then the $R$-matrix $R_{VV}$ is upper triangular with respect to the lexicographic order on the basis of $V^{\otimes 2}$ induced by the chosen basis of $V$.

\subsection{$A_{n-1}$ $\Longrightarrow$ $B_{n}$}
When $\mathfrak g=A_{n-1}$, the length of each simple root is
$(\alpha_{i},\alpha_{i})=2, 1\leq i\leq n-1$, and
$\alpha_{i}=\varepsilon_{i}-\varepsilon_{i+1}$, the corresponding
$i$-th fundamental weight is
$$\lambda_{i}=
\frac{1}{n}[(n-i)\alpha_{1}+2(n-i)\alpha_{2}+\cdots+(i-1)(n-i)\alpha_{i-1}+
i(n-i)\alpha_{i}+i(n-i-1)\alpha_{i+1}+\cdots+i\alpha_{n-1}].$$

The standard $R$-matrix for type $A$ corresponding to $\lambda_1$
satisfies the quadratic equation $(PR-qI)(PR+q^{-1}I)=0$ in
\cite{FRT1,klim}, where (here we follow the convention in Remark
2.1)
$$
R^{ij}_{kl}
=q^{\delta_{ij}}\delta_{ik}\delta_{jl}+(q-q^{-1})\delta_{il}\delta_{jk}\theta(j-i),
\quad \mbox{where~} \theta(k)= \left\{
\begin{array}{lcl}
1&~~&k>0,\\
0&~~&k\leq 0.
\end{array}
\right.
\eqno{(\star)}
$$
Moreover,
the weakly quasitriangular dual pair $(U_{q}^{ext}({\mathfrak {sl}}_{n}),\mathcal{O}_{q}(SL(n)))$ is given by
$$
\langle (m^{+})^{i}_{j},t^{k}_{l}\rangle=\lambda R^{ik}_{jl}, \quad
\langle (m^{-})^{i}_{j},t^{k}_{l}\rangle=\lambda^{-1} (R^{-1}){}^{ki}_{lj},
$$
where $\lambda=q^{-\frac{1}{n}}$.
Specially,
if we choose $R^{\prime}=P$,
then $(PR+I)(PR^{\prime}-I)=0$ for any $R$.
But when $R^{\prime}=P$,
the quadratic relations in Proposition \ref{prop1} are
$e^{i}e^{j}=\sum\limits_{a,b} R'{}^{ji}_{ab}e^{a}e^{b}=e^{i}e^{j}$,
and
$f_{i}f_{j}=\sum\limits_{a,b}f_{b}f_{a}R'{}^{ab}_{ij}=f_{i}f_{j}$.
Namely,
the resulting braided (co)vector algebras $V^{\vee}(P,R_{21}^{-1}),V(P,R)$ are the free braided (co)vector algebras $V^{\vee}(R),V(R)$.
However, in what follows, we will find that
the pair between them induced from $\langle f_{j},e^{i}\rangle=\delta_{ij}$ is degenerate.
\begin{proposition}\label{propR}
For the standard $R$-matrix $(\star)$ for type $A$,
the $q$-Serre-like relations below belong to the radicals of dual pair.
\begin{gather}
\langle f,(e^{i})^{2}e^{j}+qe^{j}(e^{i})^{2}-(1+q)e^{i}e^{j}e^{i}\rangle=0\quad(i>j), \quad
\mbox{for any} ~f\in V^{\vee}(R), \\
\langle f_{j}(f_{i})^{2}+q^{-1}(f_{i})^{2}f_{j}-(1+q^{-1})f_{i}f_{j}f_{i},e\rangle=0\quad(i>j),\quad
\mbox{for any} ~e\in V(R).
\end{gather}
\end{proposition}
\begin{proof}
The braiding $\Psi$ satisfies
$\Psi_{V,V\otimes V}=(id\otimes\Psi_{V,V})(\Psi_{V,V}\otimes id)$,
$
\Psi_{V\otimes V,V}=(\Psi_{V,V}\otimes id)(id\otimes\Psi_{V,V}).
$
The structure of braided algebras in braided (co)vector algebras $V^{\vee}(R),V(R)$ are
$$\Psi_{V,V}\circ(id\otimes \cdot)=(\cdot\otimes id)\circ\Psi_{V,V\otimes V},\quad
\Psi_{V,V}\circ(\cdot\otimes id)=(id\otimes\cdot)\circ\Psi_{V\otimes V,V}.
$$
With the value of every entry in equality $(\star)$,
we obtain the following braiding

$
\Psi(e^{k}\otimes e^{k})=qe^{k}\otimes e^{k},\, \mbox{for any } k, \quad
\Psi(e^{m}\otimes e^{n})=
\left\{
\begin{array}{ll}
e^{n}\otimes e^{m}+(q-q^{-1})e^{m}\otimes e^{n}, \ & m>n;\\
e^{n}\otimes e^{m}, \ & m<n.
\end{array}
\right.
$
\noindent
Then we have
\begin{gather*}
\underline{\Delta}(e^{k}e^{k})=1\otimes e^{k}e^{k}+(1+q)e^{k}\otimes e^{k}+e^{k}e^{k}\otimes 1,\quad \mbox{for any } k,\\
\underline{\Delta}(e^{m}e^{n})=
\left\{
\begin{array}{ll}
1\otimes e^{m}e^{n}+(1+q-q^{-1})e^{m}\otimes e^{n}+e^{n}\otimes e^{m}+e^{m}e^{n}\otimes 1,& m>n;\\
1\otimes e^{m}e^{n}+e^{m}\otimes e^{n}+e^{n}\otimes e^{m}+e^{m}e^{n}\otimes 1,& m<n.
\end{array}
\right.
\end{gather*}
Then we have
$
\langle f_{k}f_{k},e^{k}e^{k}\rangle=1+q,
$ and
$$
\langle f_{m}f_{n},e^{m}e^{n}\rangle=
\left\{
\begin{array}{ll}
1+q-q^{-1}, \ & m>n,\\
1, \ & m<n,
\end{array}
\right.
\quad
\langle f_{n}f_{m},e^{m}e^{n}\rangle=
\left\{
\begin{array}{ll}
1, \ & m>n,\\
1, \ & m<n.
\end{array}
\right.
$$

Firstly,
we will consider these elements in the quadratic homogeneous space,
which are some $q$-relations $a_{1}(e^{i})^{2}+a_{2}(e^{j})^{2}+a_{3}e^{i}e^{j}+a_{4}e^{j}e^{i}\quad (i>j)$,
where $a_{i}\in k[q,q^{-1}]$.
We only need to consider the value
$
\langle f_{a}f_{b},a_{1}(e^{i})^{2}+a_{2}(e^{j})^{2}+a_{3}e^{i}e^{j}+a_{4}e^{j}e^{i}\rangle
$
for any $a,b$.
According to
$
\langle cd,b\rangle=\langle c,b_{\underline{(1)}}\rangle\langle d,b_{\underline{(2)}}\rangle,
\langle c,ab\rangle=\langle c_{\underline{(1)}},a\rangle\langle c_{\underline{(2)}},b\rangle,
$
and
$\underline{\epsilon}(e^{i})=0,$
$\underline{\epsilon}(f_{i})=0$.
If the elements $a_{1}(e^{i})^{2}+a_{2}(e^{j})^{2}+a_{3}e^{i}e^{j}+a_{4}e^{j}e^{i}$ belong to the right radical,
then for  any $a, b$,
we obtain
$$\left.
\begin{array}{rl}
\langle f_{a}f_{b},&a_{1}(e^{i})^{2}+a_{2}(e^{j})^{2}+a_{3}e^{i}e^{j}+a_{4}e^{j}e^{i}\rangle\\
&=a_{1}(1+q)\langle f_{a},e^{i}\rangle\langle f_{b},e^{i}\rangle+a_{2}(1+q)\langle f_{a},e^{j}\rangle\langle f_{b},e^{j}\rangle\\
&\quad +\,a_{3}(1+q-q^{-1})\langle f_{a},e^{i}\rangle\langle f_{b},e^{j}\rangle+a_{3}\langle f_{a},e^{j}\rangle\langle f_{b},e^{i}\rangle\\
&\quad +\,a_{4}\langle f_{a},e^{j}\rangle\langle f_{b},e^{i}\rangle+a_{4}\langle f_{a},e^{i}\rangle\langle f_{b},e^{j}\rangle
=0.
\end{array}
\right.
$$

We only need to consider the situations $f_{a}f_{b}=f_{i}f_{i}, f_{j}f_{j}, f_{i}f_{j}, f_{j}f_{i}$, respectively,
and get the following equations
$$
\left\{
\begin{array}{l}
a_{1}(1+q)=0,\\
a_{2}(1+q)=0,\\
a_{3}(1+q-q^{-1})+a_{4}=0,\\
a_{3}+a_{4}=0.
\end{array}
\right.
$$
Solving it,
we obtain $a_{1}=a_{2}=a_{3}=a_{4}=0.$
So the dual pair in the quadratic homogeneous space is non-degenerate.
Then we will consider the cubic homogeneous space consisting of the elements
$b_{1}(e^{i})^{3}+b_{2}(e^{j})^{3}+b_{3}(e^{i})^{2}e^{j}+b_{4}e^{i}e^{j}e^{i}+b_{5}e^{j}(e^{i})^{2}$.
$$
\underline{\Delta}((e^{k})^{3})=1\otimes (e^{k})^{3}+(e^{k})^{3}\otimes 1+(1+q+q^{2})e^{k}\otimes (e^{k})^{2}
+(1+q+q^{2})(e^{k})^{2}\otimes e^{k},\quad\mbox{for any~}k.
$$
$$
\underline{\Delta}((e^{i})^{2}e^{j})=
\left\{
\begin{array}{ll}
1\otimes (e^{i})^{2}e^{j}+(q+q^{2}-q^{-1})(e^{i})^{2}\otimes e^{j}+(q+q^{2})e^{i}\otimes e^{i}e^{j}&\\
+e^{j}\otimes (e^{i})^{2}+(q-q^{-1})e^{i}\otimes e^{j}e^{i}
+(1+q)e^{i}e^{j}\otimes e^{i}+(e^{i})^{2}e^{j}\otimes 1\quad&i>j;\\
&\\
1\otimes (e^{i})^{2}e^{j}+(e^{i})^{2}\otimes e^{j}+(1+q)e^{i}\otimes e^{i}e^{j}&\\
+e^{j}\otimes (e^{i})^{2}+(1+q)e^{i}e^{j}\otimes e^{i}+(e^{i})^{2}e^{j}\otimes 1&i<j.
\end{array}
\right.
$$

With a similar method,
we get
$$
\underline{\Delta}(e^{i}e^{j}e^{i})=
\left\{
\begin{array}{ll}
1\otimes e^{i}e^{j}e^{i}+e^{i}e^{j}e^{i}\otimes 1+qe^{j}e^{i}\otimes e^{i}
+(1+q-q^{-1})(e^{i})^{2}\otimes e^{j}&\\
+qe^{i}\otimes e^{i}e^{j}+e^{j}\otimes (e^{i})^{2}
+(1+q-q^{-1})e^{i}\otimes e^{j}e^{i}+e^{i}e^{j}\otimes e^{i},\quad &i>j;\\
&\\
1\otimes e^{i}e^{j}e^{i}+e^{i}e^{j}e^{i}\otimes 1+(1+q-q^{-1})e^{j}\otimes (e^{i})^{2}&\\
+(1+q-q^{-1})e^{i}e^{j}\otimes e^{i}+qe^{i}\otimes e^{i}e^{j}+(e^{i})^{2}\otimes e^{j}+qe^{j}e^{i}\otimes e^{i},
\quad &i<j.\end{array}
\right.
$$
$$
\underline{\Delta}(e^{j}(e^{i})^{2})=
\left\{
\begin{array}{ll}1\otimes e^{j}(e^{i})^{2}+e^{j}(e^{i})^{2}\otimes 1+(e^{i})^{2}\otimes e^{j}
+(1+q)e^{i}\otimes e^{j}e^{i}&\\
+e^{j}\otimes (e^{i})^{2}+(1+q)e^{j}e^{i}\otimes e^{i},  & i>j;\\
&\\
1\otimes e^{j}(e^{i})^{2}+e^{j}(e^{i})^{2}\otimes 1+(1+q)e^{i}\otimes e^{j}e^{i}+(e^{i})^{2}\otimes e^{j}&\\
+(q+q^{2}-q^{-1})e^{j}\otimes (e^{i})^{2}+(q^{2}-1)e^{j}e^{i}\otimes e^{i}+(q-q^{-1})e^{i}e^{j}\otimes e^{i},\quad
&i<j.\end{array}
\right.
$$
If $b_{1}(e^{i})^{3}+b_{2}(e^{j})^{3}+b_{3}(e^{i})^{2}e^{j}+b_{4}e^{i}e^{j}e^{i}+b_{5}e^{j}(e^{i})^{2}$ belong to the right radical,
then for any $a, b, c$,
when $i<j$,
$$
\left.
\begin{array}{rl}
&\langle f_{a}f_{b}f_{c},b_{1}(e^{i})^{3}+b_{2}(e^{j})^{3}+b_{3}(e^{i})^{2}e^{j}+b_{4}e^{i}e^{j}e^{i}+b_{5}e^{j}(e^{i})^{2}\rangle\\
=&(1+q+q^{2})b_{1}\langle f_{a}f_{b},e^{i}e^{i}\rangle\langle f_{c},e^{i}\rangle+(1+q+q^{2})b_{2}\langle f_{a}f_{b},e^{j}e^{j}\rangle\langle f_{c},e^{j}\rangle\\
&+(b_{3}+b_{4}+b_{5})\langle f_{a}f_{b},e^{i}e^{i}\rangle\langle f_{c},e^{j}\rangle
+[qb_{4}+(q^{2}-1)b_{5}]\langle f_{a}f_{b},e^{j}e^{i}\rangle\langle f_{c},e^{i}\rangle\\
&+[(1+q)b_{3}+(1+q-q^{-1})b_{4}+(q-q^{-1})b_{5}]\langle f_{a}f_{b},e^{i}e^{j}\rangle\langle f_{c},e^{i}\rangle=0.
\end{array}
\right.
$$

Then we only need to consider the situations
$f_{a}f_{b}f_{c}=f_{i}f_{i}f_{i},
f_{j}f_{j}f_{j},
f_{i}f_{i}f_{j},
f_{j}f_{i}f_{i},
f_{i}f_{j}f_{i}$, respectively,
we obtain the following equations
$$
\left\{
\begin{array}{l}
(1+q+q^{2})(1+q)b_{1}=0,\\
(1+q+q^{2})(1+q)b_{2}=0,\\
(1+q)(b_{3}+b_{4}+b_{5})=0,\\
(1+q-q^{-1})[qb_{4}+(q^{2}-1)b_{5}]+(1+q)b_{3}+(1+q-q^{-1})b_{4}+(q-q^{-1})b_{5}=0,\\
qb_{4}+(q^{2}-1)b_{5}+(1+q)b_{3}+(1+q-q^{-1})b_{4}+(q-q^{-1})b_{5}=0.
\end{array}
\right.
$$
Solving it,
we obtain $b_{1}=b_{2}=b_{3}=b_{4}=b_{5}=0.$
We will consider the situation $i>j$.
$$
\left.
\begin{array}{rl}
&\langle f_{a}f_{b}f_{c},b_{1}(e^{i})^{3}+b_{2}(e^{j})^{3}+b_{3}(e^{i})^{2}e^{j}+b_{4}e^{i}e^{j}e^{i}+b_{5}e^{j}(e^{i})^{2}\rangle\\
=&(1+q+q^{2})b_{1}\langle f_{a}f_{b},e^{i}e^{i}\rangle\langle f_{c},e^{i}\rangle+(1+q+q^{2})b_{2}\langle f_{a}f_{b},e^{j}e^{j}\rangle\langle f_{c},e^{j}\rangle\\
&+[(q+q^{2}-q^{-1})b_{3}+(1+q-q^{-1})b_{4}+b_{5})]\langle f_{a}f_{b},e^{i}e^{i}\rangle\langle f_{c},e^{j}\rangle\\
&+[qb_{4}+(1+q)b_{5}]\langle f_{a}f_{b},e^{j}e^{i}\rangle\langle f_{c},e^{i}\rangle
+[(1+q)b_{3}+b_{4}]\langle f_{a}f_{b},e^{i}e^{j}\rangle\langle f_{c},e^{i}\rangle.
\end{array}
\right.
$$

We also only need to consider the situations
$f_{a}f_{b}f_{c}=f_{i}f_{i}f_{i},
f_{j}f_{j}f_{j},
f_{i}f_{i}f_{j},
f_{j}f_{i}f_{i},
f_{i}f_{j}f_{i}$, respectively,
and obtain the following equations
$$
\left\{
\begin{array}{l}
(1+q+q^{2})(1+q)b_{1}=0,\\
(1+q+q^{2})(1+q)b_{2}=0,\\
(1+q)[(q+q^{2}-q^{-1})b_{3}+(1+q-q^{-1})b_{4}+b_{5})]=0,\\
(1+q-q^{-1})[(1+q)b_{3}+b_{4}]+[qb_{4}+(1+q)b_{5}]=0,\\
(1+q)b_{3}+b_{4}+qb_{4}+(1+q)b_{5}=0.
\end{array}
\right.
$$
Solving it,
we get
$b_{1}=b_{2}=0,
b_{4}=-(1+q)b_{3},
b_{5}=qb_{3}.$
So only the elements
$b_{3}[(e^{i})^{2}e^{j}+qe^{j}(e^{i})^{2}-(1+q)e^{i}e^{j}e^{i}]~(i>j)$
belong to the right radical in the cubic homogenous space of $V(R)$.
With a similar analysis, we see that
only the elements
$
f_{j}(f_{i})^{2}+q^{-1}(f_{i})^{2}f_{j}-(1+q^{-1})f_{i}f_{j}f_{i}$ $(i>j)$
belong to the left radical in the cubic homogenous space of $V^{\vee}(R)$.
\end{proof}

With the above Proposition \ref{propR},
we pass to the quotients by the Hopf ideals generated by
$(e^{i})^{2}e^{j}+qe^{j}(e^{i})^{2}-(1+q)e^{i}e^{j}e^{i}$ and $f_{j}(f_{i})^{2}+q^{-1}(f_{i})^{2}f_{j}-(1+q^{-1})f_{i}f_{j}f_{i}$ $(i>j)$, respectively.
Denote by $\tilde{V}^{\vee}(R),\tilde{V}(R)$ the respective quotient braided (co)vector algebras.
With these, by the generalized double-bosonization Theorem (Theorem \ref{cor1}),
we obtain the following
\begin{theorem}
With $\lambda=q^{-\frac{1}{n}}$, identify
$e^{n},f_{n},(m^{+})^{n}_{n}c^{-1}$ with the additional simple root
vectors $E_{n}, F_{n}$ and the group-like element $K_{n}$. Then the
resulting new quantum group $U(\tilde{V}^{\vee}(R),\widetilde{U_{q}^{ext}({\mathfrak {sl}}_{n})},\tilde{V}(R))$ is exactly the
$U_{q}({\mathfrak {so}}_{2n+1})$ with $K_{i}^{\pm\frac{1}{n}}$ adjoined.
\end{theorem}
\begin{proof}
Corresponding to the vector representation of $U_{q}({\mathfrak {sl}}_{n})$,
all the diagonal and minor diagonal entries we need in the matrices $m^{\pm}$ have been given in \cite{HH1}.
For the identification in Theorem 3.1, it easily follows from Theorem \ref{cor1} that
$$
[E_{n},F_{n}]=\frac{K_{n}-K_{n}^{-1}}{q^{\frac{1}{2}}-q^{-\frac{1}{2}}},\quad
\Delta(E_{n})=E_{n}\otimes K_{n}+1\otimes E_{n},\quad
\Delta(F_{n})=F_{n}\otimes 1+K_{n}^{-1}\otimes F_{n}.
$$

We just demonstrate how to explore the cross relations between new simple root vectors $E_{n},F_{n}$ and group-like element $K_{n}$,
and the $q$-Serre relations between $E_{n} \ (F_{n})$ and $E_{n-1} \ (F_{n-1})$.
Other relations can be obtained by Theorem \ref{cor1} with a similar analysis. Note that
$$E_{n}K_{n}=e^{n}(m^{+})^{n}_{n}c^{-1}=\lambda R^{n}_{a}{}^{n}_{b}(m^{+})^{a}_{n}e^{b}c^{-1}
=\lambda R^{n}_{n}{}^{n}_{n}(m^{+})^{n}_{n}e^{n}c^{-1}=\lambda \frac{1}{\lambda}R^{n}_{n}{}^{n}_{n}(m^{+})^{n}_{n}c^{-1}e^{n}
=qK_{n}E_{n}.$$
Combining with the equalities
$e^{n-1}=e^{n}E_{n-1}-q^{-1}E_{n-1}e^{n}$ and $e^{n-1}E_{n-1}=qE_{n-1}e^{n-1}$
given in \cite{HH1},
then we get
$$(E_{n-1})^{2}E_{n}-(q+q^{-1})E_{n-1}E_{n}E_{n-1}+E_{n}(E_{n-1})^{2}=0.$$
On the other hand, combining with the equality
$(e^{n})^{2}e^{n-1}+qe^{n-1}(e^{n})^{2}=(1+q)e^{n}e^{n-1}e^{n}$, we
get
$(e^{n})^{2}(e^{n}E_{n-1}-q^{-1}E_{n-1}e^{n})+q(e^{n}E_{n-1}-q^{-1}E_{n-1}e^{n})(e^{n})^{2}
=(1+q)e^{n}(e^{n}E_{n-1}-q^{-1}E_{n-1}e^{n})e^{n}$, which can be
simplified into $
(e^{n})^{3}E_{n-1}-(q+1+q^{-1})(e^{n})^{2}E_{n-1}e^{n}+(q+1+q^{-1})e^{n}E_{n-1}(e^{n})^{2}-E_{n-1}(e^{n})^{3}=0,
$ namely,
$$
(E_{n})^{3}E_{n-1}-
\left[
\begin{array}{c}
3\\
1
\end{array}
\right]
_{q^{\frac{1}{2}}}
(E_{n})^{2}E_{n-1}E_{n}+
\left[
\begin{array}{c}
3\\
2
\end{array}
\right]
_{q^{\frac{1}{2}}}
E_{n}E_{n-1}(E_{n})^{2}-E_{n-1}(E_{n})^{3}=0.
$$

The relations between $F_{n-1}$ and $F_{n}$ can be obtained in the same way.
In view of these relations in the new quantum group,
the length of the new simple root $\alpha_{n}$ corresponding to the additional simple root vectors $E_{n},F_{n}$ is
$(\alpha_{n},\alpha_{n})=1$,
and $(\alpha_{n-1},\alpha_{n})=-1$,
$(\alpha_{j},\alpha_{n})=0$, for $1\leq j\leq n-2$.
This gives rise to the required Cartan matrix of type $B_n$ we want to have.

The proof is complete.
\end{proof}

\subsection{$A_{n-1}$ $\Longrightarrow$ $C_{n}$}
Let $V=\bigoplus\limits_{i=1}^{n}kx_{i}$ denote the representation space of
the vector representation $T_{V}$ of $U_{q}({\mathfrak {sl}}_{n})$, which is given by
$$
E_{i}(x_{j})=
\left\{
\begin{array}{ll}
x_{i+1}, &\mbox{if}~j=i,\\
0, &\mbox{otherwise},
\end{array}
\right.
\quad
F_{i}(x_{j})=
\left\{
\begin{array}{ll}
x_{i}, &\mbox{if}~j=i+1,\\
0, &\mbox{otherwise},
\end{array}
\right.
\quad
K_{i}(x_{j})=
\left\{
\begin{array}{ll}
q^{-1}x_{i}, &\mbox{if}~j=i,\\
qx_{i+1}, &\mbox{if}~j=i+1,\\
x_{j}, &\mbox{otherwise}.
\end{array}
\right.
$$
The corresponding weights of $x_{1}$ and $x_{i}$ are
$-\lambda_{1}$,
$-\lambda_{1}+\alpha_{1}+\cdots+\alpha_{i-1},i=2,\cdots,n$.
We will consider the quantum `symmetric square' of the vector representation,
denoted by $sym^{2}V$.
$sym^{2}V
=k\{x_{m}\otimes x_{m}, m=1,\cdots,n;
x_{i}\otimes x_{j}+q^{-1}x_{j}\otimes x_{i},~i<j,~i,j=1,\cdots,n\}$
is an irreducible $\frac{n(n+1)}{2}$-dimensional representation of $U_{q}({\mathfrak {sl}}_{n})$,
given by

$
E_{k}(x_{m}\otimes x_{m})=
\left\{
\begin{array}{ll}
x_{m}\otimes x_{m+1}+q^{-1}x_{m+1}\otimes x_{m}, \quad &\mbox{if}~k=m,\\
0, \quad &\mbox{otherwise},
\end{array}
\right.
$

$
F_{k}(x_{m}\otimes x_{m})=
\left\{
\begin{array}{ll}
x_{m-1}\otimes x_{m}+q^{-1}x_{m}\otimes x_{m-1}, \quad &\mbox{if}~k=m-1,\\
0, \quad &\mbox{otherwise},
\end{array}
\right.
$

$
E_{k}(x_{i}\otimes x_{j}+q^{-1}x_{j}\otimes x_{i})=
\left\{
\begin{array}{ll}
(q+q^{-1})x_{i+1}\otimes x_{i+1}, \quad &\mbox{if}~k=i,j=i+1,\\
x_{i+1}\otimes x_{j}+q^{-1}x_{j}\otimes x_{i+1},\quad &\mbox{if}~k=i,j>i+1,\\
x_{i}\otimes x_{j+1}+q^{-1}x_{j+1}\otimes x_{i},\quad &\mbox{if}~k=j,\\
0, \quad &\mbox{otherwise},
\end{array}
\right.
$

$
F_{k}(x_{i}\otimes x_{j}+q^{-1}x_{j}\otimes x_{i})=
\left\{
\begin{array}{ll}
x_{i-1}\otimes x_{j}+q^{-1}x_{j}\otimes x_{i-1},\quad &\mbox{if}~k=i-1,\\
(q+q^{-1})x_{i}\otimes x_{i},\quad &\mbox{if}~k=j-1,j=i+1,\\
x_{i}\otimes x_{j-1}+q^{-1}x_{j-1}\otimes x_{i},\quad &\mbox{if}~k=j-1,j>i+1,\\
0, \quad &\mbox{otherwise}.
\end{array}
\right.
$

For convenience,
denote by $\{\,v_i\,\}$ a basis of $sym^{2}V$ with weights $\{\,\mu_{i}\,\}$ arranged in the raising-weights order.
With this module $sym^{2}V$,
we obtain the corresponding upper triangular $R$-matrix with respect to the lexicographic order on the basis of $(sym^{2}V)^{\otimes2}$ induced
by the chosen basis of $sym^{2}V$,
denoted still by $R_{VV}$.
Note that
$
E_{i}^{2}(x_{i}\otimes x_{i})=(q+q^{-1})x_{i+1}\otimes x_{i+1},
i=1,\cdots,n-1
$
in $sym^{2}V$,
so every $E_{i}^{2}$ is not a zero action.
Moreover,
the corresponding matrix $PR_{VV}$ is non-symmetric.
If
$
(q-q^{-1})(E_{k}\otimes F_{k})(v_{i}\otimes v_{j})=f(q)\cdot(v_{m}\otimes v_{n}),
$
for
$
0\neq f(q)\in k[q,q^{-1}],
$
then
$
\mu_{m}=\mu_{i}+\alpha_{k},
$
$
\mu_{n}=\mu_{j}-\alpha_{k}.
$
Moreover,
we have
$$
E_{k}(v_{i})=
\left\{
\begin{array}{ll}
(q+q^{-1})v_{m}, \quad &\mbox{if}~ v_{i}=x_{k}\otimes x_{k+1}+q^{-1}x_{k+1}\otimes x_{k},\\
v_{m}, \quad &\mbox{otherwise,}
\end{array}
\right.
$$
which is equivalent to
$
F_{k}(v_{m})=
\left\{
\begin{array}{ll}
(q+q^{-1})v_{i}, \quad &\mbox{if}~ v_{m}=x_{k}\otimes x_{k+1}+q^{-1}x_{k+1}\otimes x_{k},\\
v_{i}, \quad &\mbox{otherwise.}
\end{array}
\right.
$

\noindent
And
$$
F_{k}(v_{j})=
\left\{
\begin{array}{ll}
(q+q^{-1})v_{n},\quad &\mbox{if}~ v_{j}=x_{k}\otimes x_{k+1}+q^{-1}x_{k+1}\otimes x_{k},\\
v_{n},\quad &\mbox{otherwise,}
\end{array}
\right.
$$
which is equivalent to
$
E_{k}(v_{n})=
\left\{
\begin{array}{ll}
(q+q^{-1})v_{j},\quad &\mbox{if}~ v_{m}=x_{k}\otimes x_{k+1}+q^{-1}x_{k+1}\otimes x_{k},\\
v_{j},\quad &\mbox{otherwise.}
\end{array}
\right.
$

In view of these facts, we obtain
$$
\left.
\begin{array}{ll}
&(q-q^{-1})(E_{k}\otimes F_{k})(v_{i}\otimes v_{j})\\
=&
\left\{
\begin{array}{ll}
av_{m}\otimes v_{n},&\mbox{if}~ v_{i}=v_{j}=x_{k}\otimes x_{k+1}+q^{-1}x_{k+1}\otimes x_{k},\\
bv_{m}\otimes v_{n}&\mbox{if}~ v_{i}=x_{k}\otimes x_{k+1}+q^{-1}x_{k+1}\otimes x_{k},
v_{j}\neq x_{k}\otimes x_{k+1}+q^{-1}x_{k+1}\otimes x_{k},\\
&\mbox{or}~
v_{i}\neq x_{k}\otimes x_{k+1}+q^{-1}x_{k+1}\otimes x_{k},
v_{j}=x_{k}\otimes x_{k+1}+q^{-1}x_{k+1}\otimes x_{k},\\
(q{-}q^{-1})v_{m}\otimes v_{n},&\mbox{otherwise,}
\end{array}
\right.
\end{array}
\right.
$$
$$
\left.
\begin{array}{ll}
&(q-q^{-1})(E_{k}\otimes F_{k})(v_{n}\otimes v_{m})\\
=&
\left\{
\begin{array}{ll}
av_{j}\otimes v_{i},&\mbox{if}~ v_{n}=v_{m}=x_{k}\otimes x_{k+1}+q^{-1}x_{k+1}\otimes x_{k},\\
bv_{j}\otimes v_{i}&\mbox{if}~ v_{n}=x_{k}\otimes x_{k+1}+q^{-1}x_{k+1}\otimes x_{k},
v_{m}\neq x_{k}\otimes x_{k+1}+q^{-1}x_{k+1}\otimes x_{k},\\
&\mbox{or}~
v_{n}\neq x_{k}\otimes x_{k+1}+q^{-1}x_{k+1}\otimes x_{k},
v_{m}=x_{k}\otimes x_{k+1}+q^{-1}x_{k+1}\otimes x_{k},\\
(q{-}q^{-1})v_{j}\otimes v_{i},&\mbox{otherwise,}
\end{array}
\right.
\end{array}
\right.
$$
where
$
a=(q-q^{-1})(q+q^{-1})^{2},
b=(q-q^{-1})(q+q^{-1}).
$
So we obtain the corresponding entries in the matrix are
$$
R_{VV}{}^{ij}_{mn}=
\left\{
\begin{array}{ll}
aq^{(\mu_{m},\mu_{n})},&\mbox{if}~ v_{i}=v_{j}=x_{k}\otimes x_{k+1}+q^{-1}x_{k+1}\otimes x_{k},\\
bq^{(\mu_{m},\mu_{n})},&\mbox{if}~ v_{i}=x_{k}\otimes x_{k+1}+q^{-1}x_{k+1}\otimes x_{k},
v_{j}\neq x_{k}\otimes x_{k+1}+q^{-1}x_{k+1}\otimes x_{k},\\
&\mbox{or}~
v_{i}\neq x_{k}\otimes x_{k+1}+q^{-1}x_{k+1}\otimes x_{k},
v_{j}=x_{k}\otimes x_{k+1}+q^{-1}x_{k+1}\otimes x_{k},\\
(q{-}q^{-1})q^{(\mu_{m},\mu_{n})},&\mbox{otherwise,}
\end{array}
\right.
$$
$$
R_{VV}{}^{n}_{j}{}^{m}_{i}=
\left\{
\begin{array}{ll}
aq^{(\mu_{i},\mu_{j})},&\mbox{if}~ v_{n}=v_{m}=x_{k}\otimes x_{k+1}+q^{-1}x_{k+1}\otimes x_{k},\\
bq^{(\mu_{i},\mu_{j})},&\mbox{if}~ v_{n}=x_{k}\otimes x_{k+1}+q^{-1}x_{k+1}\otimes x_{k},
v_{m}\neq x_{k}\otimes x_{k+1}+q^{-1}x_{k+1}\otimes x_{k},\\
&\mbox{or}~
v_{n}\neq x_{k}\otimes x_{k+1}+q^{-1}x_{k+1}\otimes x_{k},
v_{m}=x_{k}\otimes x_{k+1}+q^{-1}x_{k+1}\otimes x_{k},\\
(q{-}q^{-1})q^{(\mu_{i},\mu_{j})},&\mbox{otherwise.}
\end{array}
\right.
$$
With these,
we find that
$R_{VV}{}^{i}_{m}{}^{j}_{n}$
maybe not equal to
$R_{VV}{}^{n}_{j}{}^{m}_{i}$.
For example,
when
$
v_{i}=v_{j}=x_{k}\otimes x_{k+1}+q^{-1}x_{k+1}\otimes x_{k},
$
then
$v_{m}=x_{k+1}\otimes x_{k+1},
v_{n}=x_{k}\otimes x_{k}.
$
we obtain
\begin{gather*}
(\mu_{m},\mu_{n})=(-2\lambda_{1}+2\alpha_{1}+\cdots2\alpha_{k},-2\lambda_{1}+\cdots2\alpha_{k-1})
=4(\lambda_{1},\lambda_{1})-4,\\
(\mu_{i},\mu_{j})
=(-2\lambda_{1}+2\alpha_{1}+\cdots+2\alpha_{k-1}+\alpha_{k},-2\lambda_{1}+2\alpha_{1}+\cdots+2\alpha_{k-1}+\alpha_{k})
=4(\lambda_{1},\lambda_{1})-2.
\end{gather*}
Then the corresponding entries in the matrix $R_{VV}$ are $
R_{VV}{}^{i}_{m}{}^{j}_{n}
=(q-q^{-1})(q+q^{-1})^{2}q^{(\mu_{m},\mu_{n})}, $
$R_{VV}{}^{n}_{j}{}^{m}_{i}= (q-q^{-1})q^{(\mu_{i},\mu_{j})}$.
Obviously, $R_{VV}{}^{i}_{m}{}^{j}_{n}\neq
R_{VV}{}^{n}_{j}{}^{m}_{i} \Longleftrightarrow
(PR_{VV}){}^{j\,i}_{mn}\neq (PR_{VV}){}^{mn}_{j\,i}$. Namely,
$PR_{VV}$ is non-symmetric.

In order to obtain the dually-paired braided groups,
we need to find the matrices $R, R'$ satisfying the conditions in Proposition \ref{prop1}.
First of all, we have

\begin{proposition}\label{lemmaC}
Corresponding to the representation $sym^{2}V$,
the matrix $PR_{VV}$ obeys the minimal polynomial equation
$
(PR_{VV}-q^{\frac{4(n-1)}{n}}I)(PR_{VV}-q^{-\frac{2(n+2)}{n}}I)(PR_{VV}+q^{-\frac{4}{n}}I)=0.
$
\end{proposition}
\begin{proof}
For type $A_{n-1}$, we know that the decomposition of tensor product for this module is
$sym^{2}V\otimes sym^{2}V=V_{1}\oplus V_{2}\oplus V_{3}$ (see \cite{FH}),
where
$V_{i}$ denotes the irreducible representation with highest weight
$
2\lambda_{2},
2\lambda_{1}+\lambda_{2},
4\lambda_{1},
$ respectively.
This means that there are $3$ eigenvalues,
denoted by $y_{1},y_{2},y_{3}$.
Set $\mathcal{N}=(PR_{VV}-y_{1}I)(PR_{VV}-y_{2}I)(PR_{VV}-y_{3}I)$, and
$$
\triangle_{1}=y_{1}+y_{2}+y_{3}, \quad
\triangle_{2}=y_{1}y_{2}+y_{1}y_{3}+y_{2}y_{3},\quad
\triangle_{3}=y_{1}y_{2}y_{3}.
$$

We will consider some special rows.
In view of the representation $sym^{2}V$,
nonzero entries occurred at rows $(12)$ and $(21)$ in matrix $PR_{VV}-y_{i}I$ are
$
(PR_{VV}-y_{i}I)^{1}_{1}{}^{2}_{2}=-y_{i},
$
$
(PR_{VV}-y_{i}I)^{1}_{2}{}^{2}_{1}=q^{\frac{2(n-2)}{n}},
$
$
(PR_{VV}-y_{i}I)^{2}_{1}{}^{1}_{2}=q^{\frac{2(n-2)}{n}},
$
$
(PR_{VV}-y_{i}I)^{2}_{2}{}^{1}_{1}=q^{\frac{2(n-2)}{n}}(q+q^{-1})(q-q^{-1})-y_{i}.
$
Then the nonzero entries at row $(12)$ in matrix $\mathcal{N}$ are
\begin{gather*}
\mathcal{N}^{1}_{1}{}^{2}_{2}
=-y_{1}y_{2}y_{3}-q^{\frac{4(n-2)}{n}}(y_{1}{+}y_{2}{+}y_{3})+q^{\frac{6(n-2)}{n}}(q^2{-}q^{-2}),
\\
\mathcal{N}^{1}_{2}{}^{2}_{1}=q^{\frac{2(n-2)}{n}}(y_{1}y_{2}{+}y_{1}y_{3}{+}y_{2}y_{3})
-q^{\frac{4(n-2)}{n}}(q^2{-}q^{-2})(y_{1}{+}y_{2}{+}y_{3})
+q^{\frac{6(n-2)}{n}}(q^2{-}q^{-2})^{2}.
\end{gather*}

Nonzero entries at row $(\frac{n(n+1)}{2}-2,\frac{n(n+1)}{2})$ and $(\frac{n(n+1)}{2},\frac{n(n+1)}{2}-2)$ in $PR_{VV}-y_{i}I$ are
\begin{gather*}
(PR_{VV}-y_{i}I)^{\frac{n(n+1)}{2}-2,}_{\frac{n(n+1)}{2}-2,}{}^{\frac{n(n+1)}{2}}_{\frac{n(n+1)}{2}}=-y_{i},
\\
(PR_{VV}-y_{i}I)^{\frac{n(n+1)}{2}-2,\frac{n(n+1)}{2}}_{\frac{n(n+1)}{2},\frac{n(n+1)}{2}-2}=q^{-\frac{4}{n}},
\\
(PR_{VV}-y_{i}I)^{\frac{n(n+1)}{2},\frac{n(n+1)}{2}-2}_{\frac{n(n+1)}{2}-2,\frac{n(n+1)}{2}}=q^{-\frac{4}{n}},
\\
(PR_{VV}-y_{i}I)^{\frac{n(n+1)}{2},\frac{n(n+1)}{2}-2}_{\frac{n(n+1)}{2}-1,\frac{n(n+1)}{2}-1}=q^{\frac{2(n-2)}{n}}(q-q^{-1}),
\\
(PR_{VV}-y_{i}I)^{\frac{n(n+1)}{2},}_{\frac{n(n+1)}{2},}{}^{\frac{n(n+1)}{2}-2}_{\frac{n(n+1)}{2}-2}
=q^{\frac{n-4}{n}}(q-q^{-1})^{2}(q+q^{-1})-y_{i}.
\end{gather*}
Then the entry at row $(\frac{n(n+1)}{2}-2,\frac{n(n+1)}{2})$ and column $(\frac{n(n+1)}{2}-2,\frac{n(n+1)}{2})$ in matrix $\mathcal{N}$ is
$$
\mathcal{N}^{\frac{n(n+1)}{2}-2,}_{\frac{n(n+1)}{2}-2,}{}^{\frac{n(n+1)}{2}}_{\frac{n(n+1)}{2}}
=-y_{1}y_{2}y_{3}-q^{-\frac{8}{n}}(y_{1}+y_{2}+y_{3})+q^{\frac{n-12}{n}}(q-q^{-1})^{2}(q+q^{-1}).
$$

So we obtain the following equations
$$
\left\{
\begin{array}{l}
-y_{1}y_{2}y_{3}-q^{\frac{4(n-2)}{n}}(y_{1}{+}y_{2}{+}y_{3})+q^{\frac{6(n-2)}{n}}(q^2{-}q^{-2})=0,\\
q^{\frac{2(n-2)}{n}}(y_{1}y_{2}{+}y_{1}y_{3}{+}y_{2}y_{3})-q^{\frac{4(n-2)}{n}}(q^2{-}q^{-2})(y_{1}{+}y_{2}{+}y_{3})
+q^{\frac{6(n-2)}{n}}(q^2{-}q^{-2})^{2}=0,\\
-y_{1}y_{2}y_{3}-q^{-\frac{8}{n}}(y_{1}{+}y_{2}{+}y_{3})+q^{\frac{n-12}{n}}(q{-}q^{-1})^{2}(q{+}q^{-1})=0.
\end{array}
\right.
$$
Solving it,
we obtain that these eigenvalues are
$
q^{-\frac{2(n+2)}{n}},
q^{\frac{4(n-1)}{n}},
-q^{\frac{-4}{n}}.
$
So the minimal polynomial equation  of $PR_{VV}$ is
$
(PR_{VV}-q^{\frac{4(n-1)}{n}}I)(PR_{VV}-q^{-\frac{2(n+2)}{n}}I)(PR_{VV}+q^{-\frac{4}{n}}I)=0.
$
\end{proof}

Now, set
$
R=q^{\frac{4}{n}}R_{VV},
$
$
R^{\prime}
=RPR-(q^{-2}+q^{4})R+(q^{2}+1)P.
$
By Proposition 3.2, we have $(PR+I)(PR^{\prime}-I)=0$.
Braided (co)vector algebras $V^{\vee}(R^{\prime},R_{21}^{-1}),V(R^{\prime},R)$ are braided groups in the
braided category ${}^{\widetilde{H_{\lambda R}}}\mathfrak{M}, \ \mathfrak{M}^{\widetilde{H_{\lambda R}}}$, respectively,
and the weakly quasitriangular dual pair $(\widetilde{U_{q}^{\text{ext}}({\mathfrak {sl}}_{n})},\widetilde{ H_{\lambda R}})$ is given by
$
\langle (m^{+})^{i}_{j},t^{k}_{l}\rangle=\lambda R^{ik}_{jl},\,
\langle (m^{-})^{i}_{j},t^{k}_{l}\rangle=\lambda^{-1} (R^{-1}){}^{ki}_{lj},\,
\langle c,g\rangle=\lambda,
$
where $\lambda=q^{-\frac{4}{n}}$.
The entries in the matrices $m^{\pm}$ we need can be obtained by Lemma \ref{lem1},
listed in the following
\begin{proposition}
Corresponding to the representation $sym^{2}V$,
the diagonal and minor diagonal entries in FRT-matrix $m^{\pm}$ we need are
\begin{gather*}
(m^{+})^{1}_{2}=(q+q^{-1})(q-q^{-1})
E_{1}(m^{+})^{2}_{2},\quad
(m^{+})^{j-1}_{j}=(q-q^{-1})
E_{j-1}(m^{+})^{j}_{j}, \quad 3\leq j\leq n,
\\
(m^{+})^{i}_{i}=K^{\frac{n-2}{n}}_{1}\cdots K^{\frac{n-2(i-1)}{n}}_{i-1} K^{\frac{2(n-i)}{n}}_{i}
\cdots K^{\frac{2\cdot2}{n}}_{n-2}K^{\frac{2\cdot1}{n}}_{n-1},\quad
(m^{+})^{\frac{n(n+1)}{2}}_{\frac{n(n+1)}{2}}
=K^{-\frac{2\cdot1}{n}}_{1}K^{-\frac{2\cdot 2}{n}}_{2}\cdots K^{-\frac{2\cdot(n-1)}{n}}_{n-1},\\
(m^{-})^{i}_{i}(m^{+})^{i}_{i}=1, \quad (m^{-})^{j+1}_{j}=q(q-q^{-1})
(m^{-})^{j+1}_{j+1}F_{j},\quad 1\leq i\leq n,1\leq j\leq n-1.
\end{gather*}
\end{proposition}
With these,
we have the following
\begin{theorem}
With $\lambda=q^{-\frac{4}{n}}$,
identify $e^{\frac{n(n+1)}{2}},f_{\frac{n(n+1)}{2}},(m^{+})^{\frac{n(n+1)}{2}}_{\frac{n(n+1)}{2}}c^{-1}$ with
the additional simple root vectors $E_{n}, F_{n}$ and the group-like element $K_{n}$.
Then the resulting new quantum group $U(V^{\vee}(R_{21}^{-1},R^{\prime}),\widetilde{U_{q}^{ext}({\mathfrak {sl}}_{n})},V(R,R^{\prime}))$
is exactly the $U_{q}({\mathfrak {sp}}_{2n})$ with $K_{i}^{\pm\frac{1}{n}}$ adjoined.
\end{theorem}
\begin{proof}
For the identification in Theorem 3.2,
$[E_{n},F_{n}]=\frac{K_{n}-K_{n}^{-1}}{q^{2}-q^{-2}}$,
$\Delta(E_{n})=E_{n}\otimes K_{n}+1\otimes E_{n}$,
and
$\Delta(F_{n})=F_{n}\otimes 1+K_{n}^{-1}\otimes F_{n}$
can be deduced easily from Theorem \ref{cor1}.
Since the relations of negative part can be obtained in a similar way,
we only focus on the relations of the positive part.
$$
\begin{array}{rl}
E_{n}K_{n}
&=e^{\frac{n(n+1)}{2}}(m^{+})^{\frac{n(n+1)}{2}}_{\frac{n(n+1)}{2}}c^{-1}
=\lambda R^{\frac{n(n+1)}{2},}_{\frac{n(n+1)}{2},}{}^{\frac{n(n+1)}{2}}_{\frac{n(n+1)}{2}}
(m^{+})^{\frac{n(n+1)}{2}}_{\frac{n(n+1)}{2}}e^{\frac{n(n+1)}{2}}c^{-1}\\
&=\lambda\frac{1}{\lambda}R^{\frac{n(n+1)}{2},}_{\frac{n(n+1)}{2},}{}^{\frac{n(n+1)}{2}}_{\frac{n(n+1)}{2}}
(m^{+})^{\frac{n(n+1)}{2}}_{\frac{n(n+1)}{2}}c^{-1}e^{\frac{n(n+1)}{2}}
=q^{\frac{4}{n}}R_{VV}{}^{\frac{n(n+1)}{2},}_{\frac{n(n+1)}{2},}{}^{\frac{n(n+1)}{2}}_{\frac{n(n+1)}{2}}K_{n}E_{n}\\
&=q^{\frac{4}{n}}q^{\frac{4(n-1)}{n}}K_{n}E_{n}=q^{4}K_{n}E_{n}.
\end{array}
$$

The new group-like element
$K_{n}
=(m^{+})^{\frac{n(n+1)}{2}}_{\frac{n(n+1)}{2}}c^{-1}
=K^{-\frac{2\cdot 1}{n}}_{1}K^{-\frac{2\cdot 2}{n}}_{2}\cdots K^{-\frac{2\cdot (n-1)}{n}}_{n-1}c^{-1},
$
then when $2\leq j\leq n-2$,
we obtain

$
\begin{array}{rl}
K_{n}E_{j}&=K^{-\frac{2\cdot 1}{n}}_{1}K^{-\frac{2\cdot 2}{n}}_{2}\cdots K^{-\frac{2\cdot (n-1)}{n}}_{n-1}c^{-1}E_{j}\\
&=q^{-\frac{2(j-1)}{n}}q^{2\cdot\frac{2j}{n}}q^{-\frac{2(j+1)}{n}}
E_{j}K^{-\frac{2\cdot 1}{n}}_{1}K^{-\frac{2\cdot 2}{n}}_{2}\cdots K^{-\frac{2\cdot (n-1)}{n}}_{n-1}c^{-1}
=E_{j}K_{n},
\end{array}
$

$
\begin{array}{rl}
K_{n}E_{1}&=K^{-\frac{2\cdot 1}{n}}_{1}K^{-\frac{2\cdot 2}{n}}_{2}\cdots K^{-\frac{2\cdot (n-1)}{n}}_{n-1}c^{-1}E_{1}\\
&=q^{2\cdot\frac{2}{n}}q^{-\frac{2\cdot 2}{n}}
E_{1}K^{-\frac{2\cdot 1}{n}}_{1}K^{-\frac{2\cdot 2}{n}}_{2}\cdots K^{-\frac{2\cdot (n-1)}{n}}_{n-1}c^{-1}
=E_{1}K_{n},
\end{array}
$

$
\begin{array}{rl}
K_{n}E_{n-1}&=K^{-\frac{2\cdot 1}{n}}_{1}K^{-\frac{2\cdot 2}{n}}_{2}\cdots K^{-\frac{2\cdot (n-1)}{n}}_{n-1}c^{-1}E_{n-1}\\
&=q^{-\frac{2\cdot(n-2)}{n}}q^{2\cdot\frac{2\cdot (n-1)}{n}}
E_{n-1}K^{-\frac{2\cdot 1}{n}}_{1}K^{-\frac{2\cdot 2}{n}}_{2}\cdots K^{-\frac{2\cdot (n-1)}{n}}_{n-1}c^{-1}
=q^{2}E_{n-1}K_{n}.
\end{array}
$

In order to explore the relations between the new simple root vector $E_{n}$ and $K_{i}, 1\leq i\leq n-1$,
we consider the cross relation
$$
e^{\frac{n(n+1)}{2}}(m^{+})^{j}_{j}
=\lambda R^{j,}_{j,}{}^{\frac{n(n+1)}{2}}_{\frac{n(n+1)}{2}}(m^{+})^{j}_{j}e^{\frac{n(n+1)}{2}}
=R_{VV}{}^{j,}_{j,}{}^{\frac{n(n+1)}{2}}_{\frac{n(n+1)}{2}}(m^{+})^{j}_{j}e^{\frac{n(n+1)}{2}}
=q^{(\mu_{j},\mu_{\frac{n(n+1)}{2}})}(m^{+})^{j}_{j}e^{\frac{n(n+1)}{2}}.
$$
Combining with $K_{i}(m^{+})^{i+1}_{i+1}=(m^{+})^{i}_{i},1\leq i\leq n-1$,
which can be deduced from the equalities in the above Proposition,
we obtain
$
e^{\frac{n(n+1)}{2}}K_{i}=q^{(\mu_{i}-\mu_{i+1},\mu_{\frac{n(n+1)}{2}})}K_{i}e^{\frac{n(n+1)}{2}}
$,
which is equivalent to
$
E_{n}K_{i}=q^{(\mu_{i}-\mu_{i+1},\mu_{\frac{n(n+1)}{2}})}K_{i}E_{n}.
$
According to the weight of $x_{i}$,
we get
$\mu_{i}=-2\lambda_{1}+\alpha_{1}+\cdots+\alpha_{i-1},$
$\mu_{i+1}=-2\lambda_{1}+\alpha_{1}+\cdots+\alpha_{i-1}+\alpha_{i},$
$
\mu_{\frac{n(n+1)}{2}}=-2\lambda_{1}+2\alpha_{1}+\cdots+2\alpha_{n-2}+2\alpha_{n-1},
$
then
$$
q^{(\mu_{i}-\mu_{i+1},\mu_{\frac{n(n+1)}{2}})}
=q^{(-\alpha_{i},-2\lambda_{1}+2\alpha_{1}+\cdots+2\alpha_{n-2}+2\alpha_{n-1})}=
\left\{
\begin{array}{l}
1,\quad 1\leq i\leq n-2,\\
q^{-2},\quad i=n-1.
\end{array}
\right.
$$

So
$
K_{i}E_{n}=E_{n}K_{i}, 1\leq i\leq n-2;
$
$
E_{n}K_{n-1}=q^{2}K_{n-1}E_{n}.
$

We also observe that each $E_{i}$ is included in the entry $(m^{+})^{i}_{i+1}$ by the above Proposition,
then the $q$-Serre relations between $e^{\frac{n(n+1)}{2}}$ and $E_{i}, 1\leq i\leq n-1$ can be deduced from the following equalities
$$
\left\{
\begin{array}{l}
e^{\frac{n(n+1)}{2}}(m^{+})^{i}_{i+1}
=\lambda R^{i,\frac{n(n+1)}{2}}_{i,\frac{n(n+1)}{2}}(m^{+})^{i}_{i+1}e^{\frac{n(n+1)}{2}}
+\lambda R^{i,\frac{n(n+1)}{2}}_{i+1,\frac{n(n+1)}{2}-1}(m^{+})^{i+1}_{i+1}e^{\frac{n(n+1)}{2}-1},\\
e^{\frac{n(n+1)}{2}}(m^{+})^{i+1}_{i+1}
=\lambda R^{i+1,}_{i+1,}{}^{\frac{n(n+1)}{2}}_{\frac{n(n+1)}{2}}(m^{+})^{i+1}_{i+1}e^{\frac{n(n+1)}{2}},\\
(m^{+})^{1}_{2}=(q+q^{-1})(q-q^{-1})E_{1}(m^{+})^{2}_{2},\\
(m^{+})^{i}_{i+1}=(q-q^{-1})
E_{i}(m^{+})^{i+1}_{i+1}, \quad 2\leq i\leq n-1,\\
R^{i,\frac{n(n+1)}2}_{i+1,\frac{n(n+1)}2-1}=0,\quad 1\leq i\leq n-2.
\end{array}
\right.
$$

In virtue of these,
we obtain
$$
\left\{
\begin{array}{l}
e^{\frac{n(n+1)}{2}}E_{i}=E_{i}e^{\frac{n(n+1)}{2}},\\
e^{\frac{n(n+1)}{2}-1}=e^{\frac{n(n+1)}{2}}E_{n-1}-q^{-2}E_{n-1}e^{\frac{n(n+1)}{2}},
\end{array}
\right.
\Longleftrightarrow
\left\{
\begin{array}{l}
E_{n}E_{i}=E_{i}E_{n},\quad i<n,\\
e^{\frac{n(n+1)}{2}-1}=E_{n}E_{n-1}-q^{-2}E_{n-1}E_{n}.
\end{array}
\right.
$$

So we need to explore the relations between $e^{\frac{n(n+1)}{2}-1}$ and $e^{\frac{n(n+1)}{2}}, E_{n-1}$.
According to
$
R^{\prime}
=RPR-(q^{-2}+q^{4})R+(q^{2}+1)P,
$
we get
$R^{\prime}{}^{\frac{n(n+1)}{2},}_{\frac{n(n+1)}{2},}{}^{\frac{n(n+1)}{2}-1}_{\frac{n(n+1)}{2}-1}
=-q^{2}-1$
and
$R^{\prime}{}^{\frac{n(n+1)}{2},\frac{n(n+1)}{2}-1}_{\frac{n(n+1)}{2}-1,\frac{n(n+1)}{2}}
=q^{4}+q^{2}+1.$
Then
$$e^{\frac{n(n+1)}{2}-1}e^{\frac{n(n+1)}{2}}
=R^{\prime}{}^{\frac{n(n+1)}{2},\frac{n(n+1)}{2}-1}_{a,b}e^{a}e^{b}
=-(q^{2}+1)e^{\frac{n(n+1)}{2}}e^{\frac{n(n+1)}{2}-1}
+(q^{4}+q^{2}+1)e^{\frac{n(n+1)}{2}-1}e^{\frac{n(n+1)}{2}},
$$
so
$
e^{\frac{n(n+1)}{2}}e^{\frac{n(n+1)}{2}-1}=q^{2}e^{\frac{n(n+1)}{2}-1}e^{\frac{n(n+1)}{2}}.
$
Combining with
$e^{\frac{n(n+1)}{2}-1}=E_{n}E_{n-1}-q^{-2}E_{n-1}E_{n},$
we obtain
$$
(E_{n})^{2}E_{n-1}-(q^{2}+q^{-2})E_{n}E_{n-1}E_{n}+E_{n-1}(E_{n})^{2}=0.
$$

On the other hand,
according to
$$
e^{\frac{n(n+1)}{2}-1}(m^{+})^{n-1}_{n}
=\lambda R^{n-1,}_{n-1,}{}^{\frac{n(n+1)}{2}-1}_{\frac{n(n+1)}{2}-1}(m^{+})^{n-1}_{n}e^{\frac{n(n+1)}{2}-1}
+\lambda R^{n-1,\frac{n(n+1)}{2}-1}_{n,\frac{n(n+1)}{2}-2}(m^{+})^{n}_{n}e^{\frac{n(n+1)}{2}-2},
$$
we obtain
$
e^{\frac{n(n+1)}{2}-2}=\frac{1}{q+q^{-1}}(e^{\frac{n(n+1)}{2}-1}E_{n-1}-E_{n-1}e^{\frac{n(n+1)}{2}-1}).
$
So the relation between $e^{\frac{n(n+1)}{2}-1}$ and $E_{n-1}$
depends on the relation between $e^{\frac{n(n+1)}{2}-2}$ and $E_{n-1}$,
and the relation between $e^{\frac{n(n+1)}{2}-2}$ and $E_{n-1}$ is $e^{\frac{n(n+1)}{2}-2}E_{n-1}=q^{2}E_{n-1}e^{\frac{n(n+1)}{2}-2}$,
which is deduced from the equality
$e^{\frac{n(n+1)}{2}-2}(m^{+})^{n-1}_{n}
=\lambda R^{n-1,}_{n-1,}{}^{\frac{n(n+1)}{2}-2}_{\frac{n(n+1)}{2}-2}(m^{+})^{n-1}_{n}e^{\frac{n(n+1)}{2}-2}.$
Then we obtain
$$
(e^{\frac{n(n+1)}{2}-1}E_{n-1}-E_{n-1}e^{\frac{n(n+1)}{2}-1})E_{n-1}
=q^{2}E_{n-1}(e^{\frac{n(n+1)}{2}-1}E_{n-1}-E_{n-1}e^{\frac{n(n+1)}{2}-1}).
$$
Combining with
$e^{\frac{n(n+1)}{2}-1}=E_{n}E_{n-1}-q^{-2}E_{n-1}E_{n}$
again,
we obtain
$$
(E_{n-1})^{3}E_{n}-(q^{2}+1+q^{-2})(E_{n-1})^{2}E_{n}E_{n-1}+(q^{2}+1+q^{-2})E_{n-1}E_{n}E_{n-1}^{2}-E_{n}(E_{n-1})^{3}=0,
$$
namely,
$$
(E_{n-1})^{3}E_{n}-
\left[
\begin{array}{c}
3\\
1
\end{array}
\right]_{q}
(E_{n-1})^{2}E_{n}E_{n-1}+\left[
\begin{array}{c}
3\\
2
\end{array}
\right]_{q}E_{n-1}E_{n}E_{n-1}^{2}-E_{n}(E_{n-1})^{3}=0.
$$

In view of these relations,
the length of the new simple root $\alpha_{n}$ corresponding to the additional simple root vectors $E_{n},\,F_{n}$ is
$(\alpha_{n},\alpha_{n})=4$,
$(\alpha_{n-1},\alpha_{n})=-2$,
and
$(\alpha_{i},\alpha_{n})=0, \ 1\leq i\leq n-2$.
This gives us the Cartan matrix of type $C_n$.
The proof is complete.
\end{proof}

\subsection{$A_{n-1}$ $\Longrightarrow$ $D_{n}$}
Let $\wedge^{2}V$ denote the second quantum exterior power of the vector representation $T_{V}$ of $U_{q}({\mathfrak {sl}}_{n})$,
then
$\wedge^{2}V\cong (V\otimes V)/{sym^{2}V}=\bigoplus\limits_{i,j=1}^{n}k\{x_{i}\wedge x_{j}, i<j\}$,
and
$
x_{i}\wedge x_{i}=0,
x_{i}\wedge x_{j}=-q^{-1}x_{j}\wedge x_{i}, i<j.
$
$\wedge^{2}V$ is an irreducible  $\frac{n(n-1)}{2}$-dimensional representation of $U_{q}({\mathfrak {sl}}_{n})$,
given by

$
E_{k}(x_{i}\wedge x_{j})=
\left\{
\begin{array}{ll}
x_{i+1}\wedge x_{j}, &\mbox{if}~k=i,j>i+1,\\
x_{i}\wedge x_{j+1},&\mbox{if}~k=j,\\
0,&\mbox{otherwise},
\end{array}
\right.
$

$
F_{k}(x_{i}\wedge x_{j})=
\left\{
\begin{array}{ll}
x_{i-1}\wedge x_{j}, &\mbox{if}~k=i-1,\\
x_{i}\wedge x_{j-1},&\mbox{if}~k=j-1,j>i+1,\\
0,&\mbox{otherwise}.
\end{array}
\right.
$

For convenience,
let $\{\,v_{i}\,\}$ be a basis of $\wedge^{2}V$ with weights $\{\,\mu_{i}\,\}$ arranged in the raising-weights order.
Associated to this $\frac{n(n-1)}{2}$-dimensional module $\wedge^{2}V$,
we see that
$E_{\beta}(f_{m})=f_{n} \Longleftrightarrow F_{\beta}(f_{n})=f_{m}$
and
every $E_{i}^{2}$ is zero action.
With these,
we obtain the following
\begin{proposition}
$(1)$
The matrix $PR_{VV}$ is symmetric and obeys the minimal polynomial equation
$$
(PR_{VV}-q^{\frac{2(n-2)}{n}}I)(PR_{VV}-q^{-\frac{4}{n}}I)(PR_{VV}+q^{-\frac{4}{n}}I)=0.
$$
$(2)$
The diagonal and minor diagonal entries in FRT-matrix $m^{\pm}$ we need are
\begin{gather*}
(m^{+})^{i-1}_{i}=(q-q^{-1})E_{i}(m^{+})^{i}_{i},\quad
(m^{-})^{i}_{i-1}=q(q-q^{-1})(m^{-})^{i}_{i}F_{i},\quad 2\leq i\leq n-1.
\\
(m^{+})^{i}_{i}=K^{\frac{n-2}{n}}_{1}\cdots K^{\frac{n-2i}{n}}_{i} K^{\frac{2(n-(i+1))}{n}}_{i+1}
\cdots K^{\frac{2\cdot2}{n}}_{n-2}K^{\frac{2\cdot1}{n}}_{n-1},\quad
(m^{-})^{i}_{i}(m^{+})^{i}_{i}=1, \quad 1\leq i\leq n-1.\\
(m^{+})^{n-1}_{2n-3}=(q-q^{-1})E_{1}K_{1}^{-1}(m^{+})^{n-1}_{n-1}
=(q-q^{-1})E_{1}K^{-\frac{2}{n}}_{1}K^{\frac{n-2\cdot2}{n}}_{2}K^{\frac{n-2\cdot3}{n}}_{3}
\cdots K^{\frac{n-2\cdot(n-2)}{n}}_{n-2}K^{\frac{n-2\cdot(n-1)}{n}}_{n-1}.\\
(m^{-})^{2n-3}_{n-1}=q(q-q^{-1})(m^{-})^{n-2}_{n-2}K_{1}F_{1}
=q(q-q^{-1})K^{\frac{2}{n}}_{1}K^{-\frac{n-2\cdot2}{n}}_{2}K^{-\frac{n-2\cdot3}{n}}_{3}
\cdots K^{-\frac{n-2\cdot(n-2)}{n}}_{n-2}K^{-\frac{n-2\cdot(n-1)}{n}}_{n-1}F_{1}.\\
(m^{+})^{\frac{n(n-1)}{2}}_{\frac{n(n-1)}{2}}=K^{-\frac{2}{n}}_{1}K^{-\frac{2\cdot2}{n}}_{2}
\cdots K^{-\frac{2\cdot(n-3)}{n}}_{n-3}K^{-\frac{2\cdot(n-2)}{n}}_{n-2}K^{-\frac{n-2}{n}}_{n-1}.
\end{gather*}
\end{proposition}
\begin{proof}
These results can be obtained by the similar methods.
We only describe the minimal polynomial in $(1)$,
which is deduced by the ingenious method in Proposition 3.2.
Firstly, it is well-known that for type $A_{n-1}$,
when $n\geq 4$,
the decomposition of the tensor product of the module is
$\wedge^{2}V\otimes \wedge^{2}V=V_1\oplus V_2\oplus V_3$ (see \cite{FH}),
where $V_i$ is the irreducible representation with highest weight
$
2\lambda_{2},
\lambda_{1}+\lambda_{3},
\lambda_{4},
$ respectively.
This means that there are $3$ eigenvalues,
denoted by $y_1, y_2, y_3$.
We will also consider nonzero entries at rows $(12),(21)$ in matrix $PR_{VV}-y_{i}I$,
which are
$
(PR_{VV}-y_{i}I)^{1}_{1}{}^{2}_{2}=-y_{i},
(PR_{VV}-y_{i}I)^{1}_{2}{}^{2}_{1}=q^{\frac{n-4}{n}},
(PR_{VV}-y_{i}I)^{2}_{1}{}^{1}_{2}=q^{\frac{n-4}{n}},
(PR_{VV}-y_{i}I)^{2}_{2}{}^{1}_{1}=q^{\frac{n-4}{n}}(q-q^{-1})-y_{i}.
$
Then the entries at row $(12)$ in matrix $\mathcal{N}$ are
\begin{gather*}
\mathcal{N}^{1}_{1}{}^{2}_{2}
=-y_{1}y_{2}y_{3}-q^{\frac{2(n-4)}{n}}(y_{1}+y_{2}+y_{3})+q^{\frac{2(n-4)}{n}}(q-q^{-1}),\\
\mathcal{N}^{1}_{2}{}^{2}_{1}
=q^{\frac{n-4}{n}}(y_{1}y_{2}+y_{2}y_{3}+y_{1}y_{3})-q^{\frac{2(n-4)}{n}}(q-q^{-1})(y_{1}+y_{2}+y_{3})
+q^{\frac{3(n-4)}{n}}[1+(q-q^{-1})^{2}].
\end{gather*}

Nonzero entries at row $(1,\frac{n(n-1)}{2})$ in matrix $PR_{VV}-y_{i}I$ are
$
(PR_{VV}-y_{i}I)^{1,}_{1,}{}^{\frac{n(n-1)}{2}}_{\frac{n(n-1)}{2}}=-y_{i},
(PR_{VV}-y_{i}I)^{1,\frac{n(n-1)}{2}}_{\frac{n(n-1)}{2},1}=q^{-\frac{4}{n}}.
$
Nonzero entries at row $(\frac{n(n-1)}{2},1)$ in matrix $PR_{VV}-y_{i}I$ are
$
(PR_{VV}-y_{i}I)^{\frac{n(n-1)}{2},1}_{1,\frac{n(n-1)}{2}}=q^{-\frac{4}{n}},
(PR_{VV}-y_{i}I)^{\frac{n(n-1)}{2},}_{\frac{n(n-1)}{2},}{}^{1}_{1}=-y_{i},
$
moreover,
there exist some $a_{j},b_{j}$,
$
1<a_{j},b_{j}<\frac{n(n-1)}{2},
$
and
$
(PR_{VV}-y_{i}I)^{\frac{n(n-1)}{2},1}_{a_{j},b_{j}}=f(q),
$
for some
$
f(q)\in k[q,q^{-1}].
$
Then we obtain the entries at row $(1,\frac{n(n-1)}{2})$ in matrix $(PR_{VV}-y_{1}I)(PR_{VV}-y_{2}I)$ are
$$
\begin{array}{rl}
&[(PR_{VV}-y_{1}I)(PR_{VV}-y_{2}I)]^{1,}_{1,}{}^{\frac{n(n-1)}{2}}_{\frac{n(n-1)}{2}}\\
=&(PR_{VV}-y_{1}I)^{1,}_{1,}{}^{\frac{n(n-1)}{2}}_{\frac{n(n-1)}{2}}(PR_{VV}-y_{2}I)^{1,}_{1,}{}^{\frac{n(n-1)}{2}}_{\frac{n(n-1)}{2}}
+(PR_{VV}-y_{1}I)^{1,\frac{n(n-1)}{2}}_{\frac{n(n-1)}{2},1}
(PR_{VV}-y_{2}I)^{\frac{n(n-1)}{2},1}_{1,\frac{n(n-1)}{2}}\\
=&y_{1}y_{2}+q^{-\frac{8}{n}},
\end{array}
$$
$$
\begin{array}{rl}
&[(PR_{VV}-y_{1}I)(PR_{VV}-y_{2}I)]^{1,\frac{n(n-1)}{2}}_{\frac{n(n-1)}{2},1}\\
=&(PR_{VV}-y_{1}I)^{1,}_{1,}{}^{\frac{n(n-1)}{2}}_{\frac{n(n-1)}{2}}
(PR_{VV}-y_{2}I)^{1,\frac{n(n-1)}{2}}_{\frac{n(n-1)}{2},1}
+(PR_{VV}-y_{1}I)^{1,\frac{n(n-1)}{2}}_{\frac{n(n-1)}{2},1}
(PR_{VV}-y_{2}I)^{\frac{n(n-1)}{2},}_{\frac{n(n-1)}{2},}{}^{1}_{1}\\
=&-q^{-\frac{4}{n}}(y_{1}+y_{2}),
\end{array}
$$
$$
\begin{array}{rl}
&[(PR_{VV}-y_{1}I)(PR_{VV}-y_{2}I)]^{1,}_{a_{j},}{}^{\frac{n(n-1)}{2}}_{b_{j}}\\
=&(PR_{VV}-y_{1}I)^{1,}_{1,}{}^{\frac{n(n-1)}{2}}_{\frac{n(n-1)}{2}}(PR_{VV}-y_{2}I)^{1,}_{a_{j},}{}^{\frac{n(n-1)}{2}}_{b_{j}}
+(PR_{VV}-y_{1}I)^{1,\frac{n(n-1)}{2}}_{\frac{n(n-1)}{2},1}(PR_{VV}-y_{2}I)^{\frac{n(n-1)}{2},1}_{a_{j},b_{j}}\\
=&(PR_{VV}-y_{1}I)^{1,\frac{n(n-1)}{2}}_{\frac{n(n-1)}{2},1}(PR_{VV}-y_{2}I)^{\frac{n(n-1)}{2},1}_{a_{j},b_{j}}
=q^{-\frac{4}{n}}f(q).
\end{array}
$$
With these entries,
we obtain the entry at row $(1,\frac{n(n-1)}{2})$ and column $(1,\frac{n(n-1)}{2})$ in matrix $\mathcal{N}$ is
$$
\begin{array}{rl}
\mathcal{N}^{1,}_{1,}{}^{\frac{n(n-1)}{2}}_{\frac{n(n-1)}{2}}
=&[(PR_{VV}-y_{1}I)(PR_{VV}-y_{2}I)]^{1,}_{1,}{}^{\frac{n(n-1)}{2}}_{\frac{n(n-1)}{2}}(PR_{VV}-y_{3}I)^{1,}_{1,}{}^{\frac{n(n-1)}{2}}_{\frac{n(n-1)}{2}}\\
&+[(PR_{VV}-y_{1}I)(PR_{VV}-y_{2}I)]^{1,\frac{n(n-1)}{2}}_{\frac{n(n-1)}{2},1}
(PR_{VV}-y_{3}I)^{\frac{n(n-1)}{2},1}_{1,\frac{n(n-1)}{2}}\\
&+[(PR_{VV}-y_{1}I)(PR_{VV}-y_{2}I)]^{1,\frac{n(n-1)}{2}}_{a_{j},b_{j}}
(PR_{VV}-y_{3}I)^{a_{j},b_{j}}_{1,\frac{n(n-1)}{2}}\\
=&-(y_{1}y_{2}+q^{-\frac{8}{n}})y_{3}-q^{-\frac{4}{n}}(y_{1}+y_{2})q^{-\frac{4}{n}}+q^{-\frac{4}{n}}f(q)\cdot0\\
=&-y_{1}y_{2}y_{3}-q^{-\frac{8}{n}}(y_{1}+y_{2}+y_{3}).
\end{array}
$$
With these,
we obtain the following equations
$$
\left\{
\begin{array}{ll}
-y_{1}y_{2}y_{3}-q^{\frac{2(n-4)}{n}}(y_{1}+y_{2}+y_{3})+q^{\frac{2(n-4)}{n}}(q-q^{-1})=0,\\
q^{\frac{n-4}{n}}(y_{1}y_{2}+y_{2}y_{3}+y_{1}y_{3})-q^{\frac{2(n-4)}{n}}(q-q^{-1})(y_{1}+y_{2}+y_{3})
+q^{\frac{3(n-4)}{n}}[1+(q-q^{-1})^{2}]=0,\\
-y_{1}y_{2}y_{3}-q^{-\frac{8}{n}}(y_{1}+y_{2}+y_{3})=0.
\end{array}
\right.
$$
Solving it,
we obtain that these eigenvalues are
$
q^{\frac{2(n-2)}{n}},
\pm q^{-\frac{4}{n}}.
$
So the minimal polynomial equation  of $PR_{VV}$ is
$
(PR_{VV}-q^{\frac{2(n-2)}{n}}I)(PR_{VV}-q^{-\frac{4}{n}}I)(PR_{VV}+q^{-\frac{4}{n}}I)=0.
$
\end{proof}

Set
$
R=q^{\frac{4}{n}}R_{VV},
$
$
R^{\prime}
=RPR-(q^{2}{+}1)R+(q^{2}{+}1)P,
$
then we have $(PR{+}I)(PR^{\prime}{-}I)=0$ by Proposition 3.4,
and braided groups $V^{\vee}(R^{\prime},R_{21}^{-1}),V(R^{\prime},R)$ in the braided category
${}^{\widetilde{H_{\lambda R}}}\mathfrak{M},  \mathfrak{M}^{\widetilde{H_{\lambda R}}}$, respectively.
With these analysis,
we obtain the following

\begin{theorem}\label{theoD}
With $\lambda=q^{-\frac{4}{n}}$,
identify $e^{\frac{n(n-1)}{2}},f_{\frac{n(n-1)}{2}},(m^{+})^{\frac{n(n-1)}{2}}_{\frac{n(n-1)}{2}}c^{-1}$ with
the additional simple root vectors $E_{n}, F_{n}$ and the group-like element $K_{n}$.
Then the resulting new quantum group $U(V^{\vee}(R^{\prime},R_{21}^{-1}),\widetilde{U_{q}^{ext}({\mathfrak {sl}}_{n})},V(R^{\prime},R))$
is exactly the $U_{q}({\mathfrak {so}}_{2n})$ with $K_{i}^{\pm\frac{1}{n}}$ adjoined.
\end{theorem}
\begin{proof}
For the identification in Theorem \ref{theoD},
$[E_{n},F_{n}]=\frac{K_{n}-K_{n}^{-1}}{q-q^{-1}}$,
$\Delta(E_{n})=E_{n}\otimes K_{n}+1\otimes E_{n}$,
and
$\Delta(F_{n})=F_{n}\otimes 1+K_{n}^{-1}\otimes F_{n}$
are deduced easily from Theorem \ref{cor1}.
$$
\begin{array}{rl}
E_{n}K_{n}&=e^{\frac{n(n-1)}{2}}(m^{+})^{\frac{n(n-1)}{2}}_{\frac{n(n-1)}{2}}c^{-1}
=\lambda R^{\frac{n(n-1)}{2},}_{\frac{n(n-1)}{2},}{}^{\frac{n(n-1)}{2}}_{\frac{n(n-1)}{2}}
(m^{+})^{\frac{n(n-1)}{2}}_{\frac{n(n-1)}{2}}e^{\frac{n(n-1)}{2}}c^{-1}\\
&=\lambda\frac{1}{\lambda}R^{\frac{n(n-1)}{2},}_{\frac{n(n-1)}{2},}{}^{\frac{n(n-1)}{2}}_{\frac{n(n-1)}{2}}
(m^{+})^{\frac{n(n-1)}{2}}_{\frac{n(n-1)}{2}}c^{-1}e^{\frac{n(n-1)}{2}}
=q^{\frac{4}{n}}R_{VV}{}^{\frac{n(n-1)}{2},}_{\frac{n(n-1)}{2},}{}^{\frac{n(n-1)}{2}}_{\frac{n(n-1)}{2}}K_{n}E_{n}\\
&=q^{\frac{4}{n}}q^{\frac{2n-4}{n}}K_{n}E_{n}=q^{2}K_{n}E_{n}.
\end{array}$$

We will consider the relations between $K_{n}$ and $E_{i}, 1\leq i\leq n-1$.
According to the equality
$
(m^{+})^{\frac{n(n-1)}{2}}_{\frac{n(n-1)}{2}}=K^{-\frac{2}{n}}_{1}K^{-\frac{2\cdot2}{n}}_{2}
\cdots K^{-\frac{2\cdot(n-3)}{n}}_{n-3}K^{-\frac{2\cdot(n-2)}{n}}_{n-2}K^{-\frac{n-2}{n}}_{n-1},
$
when $1\leq j\leq n-3$,
we obtain
$$
\begin{array}{rl}
K_{n}E_{j}&=K^{-\frac{2}{n}}_{1}K^{-\frac{2\cdot2}{n}}_{2}
\cdots K^{-\frac{2\cdot(n-3)}{n}}_{n-3}K^{-\frac{2\cdot(n-2)}{n}}_{n-2}K^{-\frac{n-2}{n}}_{n-1}c^{-1}E_{j}\\
&=q^{-\frac{2\cdot(j-1)}{n}}q^{2\cdot\frac{2\cdot j}{n}}q^{-\frac{2\cdot(j+1)}{n}}E_{j}K^{-\frac{2}{n}}_{1}K^{-\frac{2\cdot2}{n}}_{2}
\cdots K^{-\frac{2\cdot(n-3)}{n}}_{n-3}K^{-\frac{2\cdot(n-2)}{n}}_{n-2}K^{-\frac{n-2}{n}}_{n-1}c^{-1}
=E_{j}K_{n},
\end{array}
$$
$$
\begin{array}{rl}
K_{n}E_{n-2}
&=K^{-\frac{2}{n}}_{1}K^{-\frac{2\cdot2}{n}}_{2}
\cdots K^{-\frac{2\cdot(n-3)}{n}}_{n-3}K^{-\frac{2\cdot(n-2)}{n}}_{n-2}K^{-\frac{n-2}{n}}_{n-1}c^{-1}E_{n-2}\\
&=q^{-\frac{2\cdot(n-3)}{n}}q^{2\cdot\frac{2\cdot (n-2)}{n}}q^{-\frac{n-2}{n}}E_{n-2}K^{-\frac{2}{n}}_{1}K^{-\frac{2\cdot2}{n}}_{2}
\cdots K^{-\frac{2\cdot(n-3)}{n}}_{n-3}K^{-\frac{2\cdot(n-2)}{n}}_{n-2}K^{-\frac{n-2}{n}}_{n-1}c^{-1}
=qE_{n-2}K_{n},
\end{array}
$$
$$
\begin{array}{rl}
K_{n}E_{n-1}
&=K^{-\frac{2}{n}}_{1}K^{-\frac{2\cdot2}{n}}_{2}
\cdots K^{-\frac{2\cdot(n-3)}{n}}_{n-3}K^{-\frac{2\cdot(n-2)}{n}}_{n-2}K^{-\frac{n-2}{n}}_{n-1}c^{-1}E_{n-1}\\
&=q^{-\frac{2\cdot(n-2)}{n}}q^{2\cdot\frac{n-2}{n}}E_{n-1}K^{-\frac{2}{n}}_{1}K^{-\frac{2\cdot2}{n}}_{2}
\cdots K^{-\frac{2\cdot(n-3)}{n}}_{n-3}K^{-\frac{2\cdot(n-2)}{n}}_{n-2}K^{-\frac{n-2}{n}}_{n-1}c^{-1}
=E_{n-1}K_{n}.
\end{array}
$$

The cross relations between the new simple root vector $E_{n}$ and $K_{i}, 1\leq i\leq n-1$ can be obtained by the following equalities
$$
\left\{
\begin{array}{l}
K_{i}(m^{+})^{i}_{i}=(m^{+})^{i-1}_{i-1},\quad 2\leq i\leq n-1,\\
e^{\frac{n(n-1)}{2}}(m^{+})^{j}_{j}
=\lambda R^{j,}_{j,}{}^{\frac{n(n-1)}{2}}_{\frac{n(n-1)}{2}}(m^{+})^{j}_{j}e^{\frac{n(n-1)}{2}}
=q^{(\mu_{j},\mu_{\frac{n(n-1)}{2}})}(m^{+})^{j}_{j}e^{\frac{n(n-1)}{2}},\quad\mbox{for any}~j.
\end{array}
\right.
$$

So
$
e^{\frac{n(n-1)}{2}}K_{i}
=q^{(\mu_{i-1}-\mu_{i},\mu_{\frac{n(n-1)}{2}})}K_{i}e^{\frac{n(n-1)}{2}}
$,
and according to
$$
q^{(\mu_{i-1}-\mu_{i},\mu_{\frac{n(n-1)}{2}})}
=q^{(-\alpha_{i},-2\lambda_{1}+2\alpha_{1}+\cdots 2\alpha_{n-2}+\alpha_{n-1})}
=\left\{
\begin{array}{l}
1,\quad 2\leq i\leq n-3,\\
q^{-1},\quad  i=n-2,\\
1,\quad i=n-1,
\end{array}
\right.
$$
then we obtain
$
\left\{
\begin{array}{l}
E_{n}K_{i}=K_{i}E_{n},\quad 2\leq i\leq n-3,\\
E_{n}K_{n-2}=q^{-1}K_{n-2}E_{n},\\
E_{n}K_{n-1}=K_{n-1}E_{n}.
\end{array}
\right.
$

Associated with
$E_{n}K_{n}=q^{2}K_{n}E_{n}$,
we obtain
$E_{n}K_{1}=K_{1}E_{n}$.
We will explore the $q$-Serre relations between $E_{n}$ and $E_{i}, 1\leq i\leq n-1$.
We see that $E_{1}$
is included in the $(m^{+})^{n-1}_{2n-3}$,
so the relation
$E_{n}E_{1}=E_{1}E_{n}$
can be deduced from
$$
e^{\frac{n(n-1)}{2}}(m^{+})^{n-1}_{2n-3}
=\lambda R^{n-1,}_{n-1,}{}^{\frac{n(n-1)}{2}}_{\frac{n(n-1)}{2}}(m^{+})^{n-1}_{2n-3}e^{\frac{n(n-1)}{2}}
=q^{(\mu_{n-1},\mu_{\frac{n(n-1)}{2}})}(m^{+})^{n-1}_{2n-3}e^{\frac{n(n-1)}{2}}.
$$
The other $q$-Serre relations can be deduced from the following relations
$$
\left\{
\begin{array}{l}
e^{\frac{n(n-1)}{2}}(m^{+})^{i-1}_{i}
=\lambda R^{i-1,}_{i-1,}{}^{\frac{n(n-1)}{2}}_{\frac{n(n-1)}{2}}(m^{+})^{i-1}_{i}e^{\frac{n(n-1)}{2}},\quad 2\leq i\leq n-1,
\mbox{and} \ i\neq n-2,\\
e^{\frac{n(n-1)}{2}}(m^{+})^{n-3}_{n-2}
=\lambda R^{n-3,}_{n-3,}{}^{\frac{n(n-1)}{2}}_{\frac{n(n-1)}{2}}(m^{+})^{n-3}_{n-2}e^{\frac{n(n-1)}{2}}
+\lambda R^{n-3,}_{n-2,}{}^{\frac{n(n-1)}{2}}_{\frac{n(n-1)}{2}}(m^{+})^{n-2}_{n-2}e^{\frac{n(n-1)}{2}-1},\\
(m^{+})^{i-1}_{i}=(q-q^{-1})E_{i}(m^{+})^{i}_{i},\quad 2\leq i\leq n-1.
\end{array}
\right.
$$
Then we obtain
$
\left\{
\begin{array}{l}
E_{n}E_{i}=E_{i}E_{n},\quad 2\leq i\leq n-1,\mbox{and} \ i\neq n-2,\\
e^{\frac{n(n-1)}{2}-1}
=E_{n}E_{n-2}-q^{-1}E_{n-2}E_{n}.
\end{array}
\right.
$

In order to explore the relation of $e^{\frac{n(n-1)}{2}}$ and $E_{n-2}$,
we need to know the relations between $e^{\frac{n(n-1)}{2}}$ and $e^{\frac{n(n-1)}{2}-1}, E_{n-2}$.
From
$
R^{\prime}
=RPR-(q^{2}+1)R+(q^{2}+1)P,
$
we get
$$
R^{\prime}{}^{\frac{n(n-1)}{2},}_{\frac{n(n-1)}{2}-1,}{}^{\frac{n(n-1)}{2}-1}_{\frac{n(n-1)}{2}}=2q^{2}+1,\quad
R^{\prime}{}^{\frac{n(n-1)}{2},}_{\frac{n(n-1)}{2},}{}^{\frac{n(n-1)}{2}-1}_{\frac{n(n-1)}{2}-1}=-2q,
$$
then
$
e^{\frac{n(n-1)}{2}-1}e^{\frac{n(n-1)}{2}}
=R^{\prime}{}^{\frac{n(n-1)}{2},}_{a,}{}^{\frac{n(n-1)}{2}-1}_{b}e^{a}e^{b}
=(2q^{2}+1)e^{\frac{n(n-1)}{2}-1}e^{\frac{n(n-1)}{2}}
-2qe^{\frac{n(n-1)}{2}}e^{\frac{n(n-1)}{2}-1},
$
so
$
e^{\frac{n(n-1)}{2}}e^{\frac{n(n-1)}{2}-1}
=qe^{\frac{n(n-1)}{2}-1}e^{\frac{n(n-1)}{2}}.
$
Combining with
$
e^{\frac{n(n-1)}{2}-1}
=E_{n}E_{n-2}-q^{-1}E_{n-2}E_{n},
$
we obtain
$$
(E_{n})^{2}E_{n-2}-(q+q^{-1})E_{n}E_{n-2}E_{n}+E_{n-2}(E_{n})^{2}=0.
$$
On the other hand,
the relation between $e^{\frac{n(n-1)}{2}}$ and $E_{n-2}$ is
$
e^{\frac{n(n-1)}{2}-1}E_{n-2}=qE_{n-2}e^{\frac{n(n-1)}{2}-1},
$
which is deduced from
$
e^{\frac{n(n-1)}{2}-1}(m^{+})^{n-3}_{n-2}
=\lambda R^{n-3,}_{n-3,}{}^{\frac{n(n-1)}{2}-1}_{\frac{n(n-1)}{2}-1}(m^{+})^{n-3}_{n-2}e^{\frac{n(n-1)}{2}-1}.
$
Combining with
$
e^{\frac{n(n-1)}{2}-1}
=E_{n}E_{n-2}-q^{-1}E_{n-2}E_{n}
$
again,
we obtain
$$
(E_{n-2})^{2}E_{n}-(q+q^{-1})E_{n-2}E_{n}E_{n-2}+E_{n}(E_{n-2})^{2}=0.
$$
The relations of negative part can be obtained by the similar analysis.

The length of the new simple root $\alpha_{n}$ corresponding to the additional simple root vector $E_{n},F_{n}$ is
$(\alpha_{n},\alpha_{n})=2$,
$(\alpha_{n-1},\alpha_{n})=0$,
$(\alpha_{n-2},\alpha_{n})=-1$,
and
$(\alpha_{i},\alpha_{n})=0, 1\leq i\leq n-3$.
This leads to the Cartan matrix of type $D_n$.

The proof is complete.
\end{proof}

\section{How the tree of quantum groups grows up by double-bosonization procedure}
As a summary of the results in section 3, together with the results of \cite{HH1, HH2, HH3} we have obtained,
we also have a root-system expression for the double-bosonization recursive constructions of $U_q(\mathfrak g)$'s for the finite-dimensional complex simple Lie algebras $\mathfrak{g}$.

\begin{proposition}
Let $T_{V}$ be an irreducible $p$-dimensional representation of
$U_{q}(\mathfrak g)$ with highest weight $-\mu$ (the difference here
from \cite{rosso} results from that we use the Majid's version of
$U_q(\mathfrak g)$ whose differences with its standard version lie
$K_i^{\pm}\mapsto K_i^{\mp}$ and $E_i\mapsto -E_i$), $\nu$ the
weight of the central element $c^{-1}$,
$(a_{ij})_{(n-1)\times(n-1)}$ the Cartan matrix of $\mathfrak g$.
Then the corresponding Cartan matrix of the new quantum group
$U(V^{\vee}(R^{\prime},
R_{21}^{-1}),\widetilde{U_{q}^{ext}(\mathfrak g)},V(R^{\prime},R))$
is of a higher rank one, which is obtained from
$(a_{ij})_{(n-1)\times(n-1)}$ by adding a row and a column with:
$a_{i,n}=\frac{2(\alpha_{i},\mu)}{(\alpha_{i},\alpha_{i})}, \,
a_{n,i}=\frac{2(\mu,\alpha_{i})}{(\mu,\mu)+(\nu,\nu)}$, and $\nu$ is
orthogonal to $\mu$ and $\alpha_{i},i=1,\cdots,n-1$.
\end{proposition}
\begin{proof}
From the recursive constructions in \cite{HH1,HH2,HH3} and the results in section 3 of
this paper, we know that the new additional group-like element
$K_{n}$ is the element $(m^{+})^{p}_{p}c^{-1}$. From the
expression of $(m^{+})^{p}_{p}$ in each case, we observe that the
corresponding weight of $(m^{+})^{p}_{p}$ is $\mu$, namely
$(m^{+})^{p}_{p}=K_{\mu}$, so $K_{n}=K_{\mu+\nu}$. Every simple
root vector $E_{j}$ normally locates in the minor diagonal entry
$(m^{+})^{i}_{i+1}$, and
$(m^{+})^{i}_{i+1}=a_{i}E_{j}(m^{+})^{i+1}_{i+1}$, $a_{i}\in
k[q,q^{-1}]$. According to the cross relations in Theorem
\ref{cor1}, we obtain the following relations between $e^{i}$ and
$e^{i-1}$.
$$ \left\{
\begin{array}{l}
e^{i}(m^{+})^{i}_{i+1}=\lambda R^{i}_{a}{}^{i}_{b}(m^{+})^{a}_{i+1}e^{b}
=\lambda R^{i}_{i}{}^{i}_{i}(m^{+})^{i}_{i+1}e^{i}+\lambda R^{i}_{i+1}{}^{i}_{i-1}(m^{+})^{i+1}_{i+1}e^{i-1},\\
e^{i}(m^{+})^{i+1}_{i+1}=\lambda R^{i+1}_{i+1}{}^{i}_{i}(m^{+})^{i+1}_{i+1}e^{i},\\
e^{i-1}(m^{+})^{i+1}_{i+1}=\lambda
R^{i+1}_{i+1}{}^{i-1}_{i-1}(m^{+})^{i+1}_{i+1}e^{i-1}.
\end{array}
\right.
$$
Combining with $(m^{+})^{i}_{i+1}=a_{i}E_{j}(m^{+})^{i+1}_{i+1}$, we
get
$e^{i-1}=a_{i}\frac{R^{i+1}_{i+1}{}^{i-1}_{i-1}}{R^{i}_{i+1}{}^{i}_{i-1}}
(e^{i}E_{j}-\frac{R^{i}_{i}{}^{i}_{i}}{R^{i+1}_{i+1}{}^{i}_{i}}E_{j}e^{i})$.
Then $\text{wt}(e^{i-1})=\text{wt}(e^{i})+\alpha_j$, combining with
$e^{k}\lhd c^{-1}=\lambda^{-1}e^{k}$ for any $k$, so
$(\nu,\alpha_{j})=0$. On the other hand, we know the explicit form
of the additional group-like element $K_{\mu+\nu}$ in each recursive
construction, so we can obtain the specific form of $\nu$, and $\nu$
is orthogonal to $\mu$.
\end{proof}

In this way, we will give the specific form of weights $\mu$ and $\nu$ for the above three cases in section 3.
The imaginary line and the filled circle mean that the Dynkin diagram $A_{n-1}$ extends to
the new added simple root with an arrow pointing to the shorter of the two roots.

$(1)$ The $B_{n}$ series.
\ Take $\mathfrak{g}={\mathfrak {sl}}_{n}$ and $V$ the first fundamental representation (i.e., the vector representation),
with lowest weight $\mu=-\lambda_{n-1}$,
choose $\nu=\frac{1}{n}\sum\limits_{i=1}^{n}\varepsilon_{i}$.
Then we get ${\mathfrak {so}}_{2n+1}$.

\setlength{\unitlength}{1mm}
\begin{picture}(98,10)
\put(6,4){\circle{1}}
\put(6.5,4){\line(1,0){12}}
\put(3,0){$\varepsilon_{1}-\varepsilon_{2}$}
\put(5,6){$\alpha_{1}$}
\put(19,4){\circle{1}}
\put(19.5,4){\line(1,0){12}}
\put(18,6){$\alpha_{2}$}
\put(15,0){$\varepsilon_{2}-\varepsilon_{3}$}
\multiput(32.5,4)(3,0){4}{\line(1,0){2}}
\put(45.5,4){\line(1,0){12}}
\put(58,4){\circle{1}}
\put(58.5,4){\line(1,0){12}}
\put(57,6){$\alpha_{n-2}$}
\put(43,0){$\varepsilon_{n-2}-\varepsilon_{n-1}$}
\put(71,4){\circle{1}}
\multiput(71.5,3.7)(3,0){6}{\line(1,0){2}}
\multiput(71.5,4.3)(3,0){6}{\line(1,0){2}}
\put(80,3.1){$>$}
\put(70,6){$\alpha_{n-1}$}
\put(65,0){$\varepsilon_{n-1}-\varepsilon_{n}$}
\put(89,4){\circle*{1}}
\put(89,0){$\varepsilon_{n}$}
\end{picture}

$(2)$ The $C_{n}$ series.
\ Take $\mathfrak{g}={\mathfrak {sl}}_{n}$ and $sym^{2}V$ the quantum symmetric square of vector representation,
with lowest weight $\mu=-2\lambda_{n-1}$,
choose $\nu=\frac{2}{n}\sum\limits_{i=1}^{n}\varepsilon_{i}$.
Then we get ${\mathfrak {sp}}_{2n}$.
\setlength{\unitlength}{1mm}
\begin{picture}(98,10)
\put(6,4){\circle{1}}
\put(6.5,4){\line(1,0){12}}
\put(3,0){$\varepsilon_{1}-\varepsilon_{2}$}
\put(5,6){$\alpha_{1}$}
\put(19,4){\circle{1}}
\put(19.5,4){\line(1,0){12}}
\put(18,6){$\alpha_{2}$}
\put(15,0){$\varepsilon_{2}-\varepsilon_{3}$}
\multiput(32.5,4)(3,0){4}{\line(1,0){2}}
\put(45.5,4){\line(1,0){12}}
\put(58,4){\circle{1}}
\put(58.5,4){\line(1,0){12}}
\put(57,6){$\alpha_{n-2}$}
\put(43,0){$\varepsilon_{n-2}-\varepsilon_{n-1}$}
\put(71,4){\circle{1}}
\multiput(71.5,3.7)(3,0){6}{\line(1,0){2}}
\multiput(71.5,4.3)(3,0){6}{\line(1,0){2}}
\put(80,3.1){$<$}
\put(70,6){$\alpha_{n-1}$}
\put(65,0){$\varepsilon_{n-1}-\varepsilon_{n}$}
\put(89,4){\circle*{1}}
\put(89,0){$2\varepsilon_{n}$}
\end{picture}

$(3)$ The $D_{n}$ series.
\ Take $\mathfrak{g}={\mathfrak {sl}}_{n}$ and $\wedge^{2}V$ the second quantum exterior power of vector representation,
with lowest weight $\mu=-\lambda_{n-2}$,
choose $\nu=\frac{2}{n}\sum\limits_{i=1}^{n}\varepsilon_{i}$.
Then we get ${\mathfrak {so}}_{2n}$.
\setlength{\unitlength}{1mm}
\begin{picture}(98,13)
\put(6,4){\circle{1}}
\put(6.5,4){\line(1,0){12}}
\put(3,0){$\varepsilon_{1}-\varepsilon_{2}$}
\put(5,6){$\alpha_{1}$}
\put(19,4){\circle{1}}
\put(19.5,4){\line(1,0){12}}
\put(18,6){$\alpha_{2}$}
\put(15,0){$\varepsilon_{2}-\varepsilon_{3}$}
\multiput(32.5,4)(3,0){4}{\line(1,0){2}}
\put(45.5,4){\line(1,0){12}}
\put(58,4){\circle{1}}
\multiput(58.5,4)(3,0){5}{\line(1,0){2}}
\put(57,6){$\alpha_{n-2}$}
\put(45,0){$\varepsilon_{n-2}-\varepsilon_{n-1}$}
\put(73,4){\circle*{1}}
\put(71,1){$\varepsilon_{n-1}+\varepsilon_{n}$}
\put(70,13){$\alpha_{n-1}$}
\put(72,8){$\varepsilon_{n-1}-\varepsilon_{n}$}
\put(72,11){\circle{1}}
\put(58.5,4){\line(2,1){13}}
\end{picture}

In order to give the readers an intuitive understanding, based on our constructions in section 3 as well as those in \cite{HH1, HH2, HH3},
we draw the tree of quantum groups with nodes $U_{q}(\mathfrak{g})$'s
for all the finite dimensional complex simple Lie algebras $\mathfrak{g}$,
which grow out of the source node $A_1$ inductively by a series of suitably chosen double-bosonization procedures.
This is the Majid's main expectation for his conjeture (\cite{majid1, majid6}).

\noindent
\includegraphics*{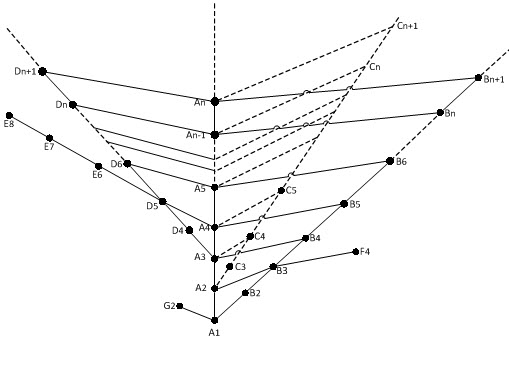}


\end{document}